\documentclass[11pt,twoside]{article}

\usepackage{amsmath,amssymb,amsthm,mathtools}
\usepackage[a4paper,margin=1in]{geometry}
\usepackage{microtype}
\allowdisplaybreaks[2]
\usepackage[hidelinks]{hyperref}
\usepackage{enumitem}
\usepackage{microtype}
\usepackage[nameinlink,capitalise]{cleveref}
\usepackage{xcolor}
\usepackage{booktabs}
\usepackage{array}
\usepackage{mathrsfs}

\numberwithin{equation}{section}

\newtheorem{theorem}{Theorem}[section]
\newtheorem{proposition}[theorem]{Proposition}
\newtheorem{lemma}[theorem]{Lemma}
\theoremstyle{remark}
\newtheorem{remark}[theorem]{Remark}

\newcommand{\R}{\mathbb R}
\newcommand{\cC}{\mathcal C}
\newcommand{\cN}{\mathcal N}
\newcommand{\cQ}{\mathcal Q}
\newcommand{\DGN}{\Delta_{\mathrm{GN}}}
\newcommand{\eps}{\varepsilon}
\newcommand{\OO}{\mathcal O}
\newcommand{\supp}{\operatorname{supp}}

\newcommand{\dd}{\,\mathrm d}

\numberwithin{equation}{section}

\begin{document}

\title{Density Collapse and Gradient Catastrophe in a One-Dimensional Euler--Poisson--Cattaneo System}
\author{{\sc Qingsong Zhao}\thanks{School of Mathematics and Science, Nanyang Institute of Technology, Nanyang 473004, China. Email: qqsszhao@nyist.edu.cn}}
\date{}
\maketitle

\begin{abstract}
We study finite-time singularity formation for a one-dimensional Euler--Poisson system in Lagrangian coordinates with Cattaneo heat conduction and quadratic heat-flux corrections in the pressure and internal energy. Under the structural choice $\kappa=\kappa_0v$, two different breakdown mechanisms are obtained. First, an explicit affine solution reaches $v=0$ in finite time, so the Eulerian density $\rho=1/v$ diverges while the remaining state variables stay finite at each fixed spatial point. Second, for $a'(1)<0$, the Poisson field is localized near equilibrium by $F=\phi_x/v$, producing a strictly hyperbolic balance law subject to the propagated constraint $F_x=1-v$. We compute the characteristic speeds and an explicit genuine-nonlinearity coefficient and show that genuine nonlinearity fails for all propagating families only when two explicit algebraic degeneracy conditions hold simultaneously. For neutral compactly supported data satisfying the Poisson constraint exactly, a narrow compressive wave generates only an $O(\varepsilon^2\eta)$ zero-speed Poisson mode. A constraint-compatible John--H\"ormander--B\"arlin bootstrap and a Riccati comparison then yield finite-time gradient catastrophe while the solution remains uniformly close to equilibrium and $F_x$ stays bounded. The two results exhibit density collapse and small-amplitude gradient blow-up as distinct breakdown mechanisms within the same Euler--Poisson--Cattaneo model.

\\[2mm]
\end{abstract}

\section{Introduction}\label{sec:intro}

Finite-time singularity formation in compressible flow reflects a competition between nonlinear wave steepening and the mechanisms that transport or dissipate energy.  This competition changes qualitatively when Fourier heat conduction is replaced by Cattaneo's law: the heat flux becomes a dynamical variable with a finite relaxation time, and in the hyperbolic formulations considered below the thermal component enters the characteristic dynamics instead of furnishing Fourier-type parabolic smoothing.  Coupling the fluid to a self-consistent Poisson field introduces a second feature: the force is tied to the density through a differential constraint, so admissible perturbations no longer fill the whole phase space.  It is therefore natural to ask whether nonlinear compression can still produce finite-time singularities under this constraint, and whether the resulting breakdown occurs through the state itself or only through its derivatives; see \cite{Cattaneo1948,KawashimaUeda2016,HuRacke2016,HuRackeWang2022,Racke2025,Zhao2026HNS} for related Cattaneo-type fluid models.

Motivated by these questions, we study the one-dimensional system in Lagrangian coordinates
\begin{subequations}\label{CEP}
\begin{align}
v_t-u_x&=0, \label{CEP-a}\\
u_t+p_x&=\frac{\phi_x}{v}, \label{CEP-b}\\
e_t+pu_x+q_x&=0,\label{CEP-c}\\
\tau q_t+q+\frac{\kappa\theta_x}{v}&=0,\label{CEP-d}\\
\left(\frac{\phi_x}{v}\right)_x&=1-v.\label{CEP-e}
\end{align}
\end{subequations}
The constitutive relations are
\begin{equation}\label{pressure}
p=a(v)+\alpha(1-\theta)-\frac{\tau q^2}{2\kappa\theta}.
\end{equation}
\begin{equation}\label{energy}
e=C_v\theta+P(v)-\alpha(v-1)+\frac{\tau vq^2}{\kappa\theta}.
\end{equation}
\begin{equation}\label{Pdef}
P(v)=-\int_1^v a(\xi)\,\dd\xi.
\end{equation}
and throughout the paper
\begin{equation}\label{kappa}
\kappa=\kappa_0v,\qquad \kappa_0>0,
\end{equation}
with
\begin{equation}\label{equil-assump}
a(1)=0,\qquad C_v>0,\quad \alpha>0,\quad \tau>0.
\end{equation}
The equilibrium is
\[
(v,u,\theta,q)=(1,0,1,0).
\]
The Eulerian density is $\rho=1/v$.

The constitutive structure in \eqref{pressure}--\eqref{kappa} is chosen so that the heat flux remains a dynamical variable rather than being eliminated through Fourier's law.  In particular, the quadratic $q$-corrections in $p$ and $e$ are of the same extended-thermodynamic form as the heat-flux corrections derived for a related hyperbolic Navier--Stokes closure in \cite{Zhao2026HNS}; entropy and dissipative structures for the Euler--Cattaneo--Maxwell system were studied in \cite{KawashimaUeda2016}.  The special conductivity \eqref{kappa} is also analytically decisive: it reduces the Cattaneo equation to
\[
\tau q_t+q+\kappa_0\theta_x=0,
\]
and removes the explicit $v$-dependence from the quadratic heat-flux contribution to the internal energy.  Thus, near equilibrium, the thermal relaxation participates in the hyperbolic characteristic structure instead of supplying a parabolic smoothing term.

The present paper focuses on two complementary questions.  First, can the solution itself reach the boundary $v=0$ of the physical state space in finite time, so that the Eulerian density becomes infinite?  Second, if the solution starts with arbitrarily small amplitude near equilibrium and remains in the interior of the physical state space, can nonlinear compression nevertheless force first derivatives to blow up in finite time?  The first question concerns a large-scale state-space collapse; the second concerns a genuinely hyperbolic gradient catastrophe.  Showing that both occur in the same system is the main theme of the paper.

For the gradient-catastrophe result it is convenient to introduce the electric-field variable
\begin{equation}\label{Fdef-intro}
F:=\frac{\phi_x}{v}.
\end{equation}
The Poisson equation then imposes the differential constraint $F_x=1-v$; for neutral compact perturbations this variable will satisfy $F_t=-u$.  This localization, together with the constraint it carries, is the basic structural feature used below.

\subsection{Main results}\label{sec:model-main}

The first result exhibits a state-space collapse that is completely explicit.

\begin{theorem}[Explicit density collapse]\label{thm:affine}
Let $\mu>1$ and assume that $a\in C^1((0,1])$. Then
\begin{equation}\label{affine-sol}
\begin{aligned}
v(t,x)&=1-\mu\sin t,\\
u(t,x)&=-\mu\cos t\,x,\\
\theta(t,x)&=\exp\!\left(-\frac{\alpha \mu}{C_v}\sin t\right),\\
q(t,x)&=0,
\end{aligned}
\end{equation}
together with
\begin{equation}\label{affine-Fphi}
F(t,x)=\frac{\phi_x}{v}=\mu\sin t\,x,
\end{equation}
for example with
\[
\phi(t,x)=\frac12\bigl(1-\mu\sin t\bigr)\mu\sin t\,x^2
\]
(up to an additive function of $t$), solves \eqref{CEP} classically for
\[
0\le t<T_c:=\arcsin(\mu^{-1}).
\]
Moreover,
\[
v(t,x)\downarrow0,
\qquad
\rho(t,x)=v(t,x)^{-1}\uparrow+\infty
\quad\text{as }t\uparrow T_c,
\]
whereas, for each fixed $x$, the variables $u,\theta,q,F$ remain finite up to $T_c$.
\end{theorem}

The solution in Theorem~\ref{thm:affine} is an affine whole-line solution, not a finite-energy perturbation of equilibrium: in particular, $u(t,x)$ grows linearly in $x$.  Its role is to exhibit explicitly a state-space collapse mechanism.  A conditional compact localization, under an additional finite-propagation hypothesis along the affine trajectory, is discussed in Remark~\ref{rem:affine-localize}.

The second result concerns a different mechanism.  For the localized system associated with \eqref{Fdef-intro}, let $\lambda_p$ and $r_p$ denote a nonzero characteristic speed and a corresponding right eigenvector.  We shall use the structural condition
\begin{equation}\tag{GN}\label{GN-assump}
\nabla\lambda_p(\bar U)\cdot r_p(\bar U)\ne0
\quad\text{for at least one nonzero characteristic family }p,
\end{equation}
where $\bar U=(1,0,1,0,0)^T$.  Subsection~\ref{sec:hyperbolicity} gives an explicit algebraic criterion equivalent to \eqref{GN-assump} and shows that failure of \eqref{GN-assump} requires two explicit algebraic degeneracy conditions to hold simultaneously.

\begin{theorem}[Finite-time gradient catastrophe]\label{thm:gradient}
Assume that $a\in C^4(I)$ on some open interval $I$ containing $v=1$, that
\[
a'(1)<0,
\]
and that \eqref{GN-assump} holds. Then, for every sufficiently small $\delta_0>0$, there exist compactly supported perturbations
\[
(v_0-1,u_0,\theta_0-1,q_0,F_0)\in C_c^3(\R)
\]
satisfying
\begin{equation}\label{compat-main}
F_{0x}=1-v_0,
\qquad
\int_\R(v_0-1)\,\dd x=0,
\end{equation}
and
\[
\|(v_0-1,u_0,\theta_0-1,q_0,F_0)\|_{L^\infty}<\delta_0,
\]
for which the corresponding maximal $H^3$ classical solution has a finite lifespan $T_{\max}<\infty$.  Moreover,
\begin{equation}\label{small-amplitude}
\sup_{0\le t<T_{\max}}\|(v-1,u,\theta-1,q,F)(t)\|_{L^\infty}\le C\delta_0,
\end{equation}
where $C$ is independent of $\delta_0$, while
\begin{equation}\label{grad-blow-main}
\limsup_{t\uparrow T_{\max}}
\bigl(\|v_x(t)\|_{L^\infty}+\|u_x(t)\|_{L^\infty}
+\|\theta_x(t)\|_{L^\infty}+\|q_x(t)\|_{L^\infty}\bigr)=+\infty.
\end{equation}
Moreover, the propagated constraint gives
\[
\sup_{t<T_{\max}}\|F_x(t)\|_{L^\infty}\le C\delta_0.
\]
If, in addition, $a$ is $C^\infty$ near $v=1$, the initial perturbation may be chosen in $C_c^\infty(\R)$.
\end{theorem}

\begin{remark}[Small amplitude versus small Sobolev norm]\label{rem:small-amplitude}
Theorem~\ref{thm:gradient} is a small-amplitude, short-wave singularity result, not a small-$H^3$ result.  In the construction below the physical pulse has amplitude $O(\eps)$ and width $\ell=\eps\eta$, while its dominant initial gradient is $O(\eta^{-1})$.  More generally, a $k$th physical derivative of the fluid--thermal profile has the scaling
\[
\|\partial_x^k(U_0-\bar U)\|_{L^2}
=O\!\left(\eps^{3/2-k}\eta^{1/2-k}\right)
\]
for the fixed profiles used in the proof.  Thus the data need not be small in $H^3$, and the result does not conflict with global theories that assume small Sobolev norms.
\end{remark}

Thus the two theorems describe distinct singularity mechanisms.  In Theorem~\ref{thm:affine}, the specific volume itself reaches the boundary $v=0$ of the physical state space.  In Theorem~\ref{thm:gradient}, the state remains uniformly close to equilibrium and the breakdown occurs only through fluid--thermal derivatives.  In particular, for sufficiently small $\delta_0$, \eqref{small-amplitude} keeps $v$ and $\theta$ bounded away from zero, preserves positivity of
\[
\cC=C_v-\frac{\tau q^2}{\kappa_0\theta^2},
\]
and keeps the system strictly hyperbolic up to $T_{\max}$.

\subsection{Main difficulties and new ideas}\label{sec:novelty-strategy}

It is useful to separate the part of the argument that belongs to the general theory of hyperbolic balance laws from the part that is specific to the Euler--Poisson--Cattaneo system.  If the five components of the localized unknown could be prescribed independently, then the general result of B\"arlin \cite{Barlin2023} would strongly suggest finite-time gradient blow-up once strict hyperbolicity and genuine nonlinearity are verified.  We do not invoke that theorem as a black box: its pure-wave data are not automatically compatible with the Poisson constraint, so the characteristic estimates have to be reopened and shown to remain stable under the constraint-generated zero mode.  The present problem is different because the electric-field component is not independent: physically admissible data must satisfy the differential constraint and neutrality condition
\begin{equation}\label{intro-constraint}
F_{0x}=1-v_0,
\qquad
\int_{\R}(v_0-1)\,\dd x=0.
\end{equation}
A compactly supported pure simple wave for the unconstrained balance law generally violates \eqref{intro-constraint}.  Thus the main issue is not merely to produce a compressive genuinely nonlinear wave, but to produce one \emph{on the Poisson constraint manifold} and to prove that the correction required by the constraint does not destroy the Riccati steepening mechanism.

The proof of Theorem~\ref{thm:gradient} contains four model-specific ingredients.
\begin{enumerate}
\item \emph{Localization and propagation of the Poisson constraint.}
With $F=\phi_x/v$, the elliptic coupling is converted into a local zero-speed mode satisfying $F_t=-u$, while
\[
F_x=1-v
\]
is propagated exactly by the evolution.  The principal matrix therefore consists of four propagating fluid--thermal modes together with one Poisson mode of speed zero.

\item \emph{An explicit genuine-nonlinearity criterion.}
We compute the full characteristic structure near equilibrium and, in particular, the coefficient $\nabla\lambda_p\cdot r_p$.  The $q$-direction contribution is nonzero and must be retained.  The resulting algebraic criterion shows that the loss of genuine nonlinearity for all propagating families requires two explicit degeneracy conditions simultaneously; see Subsection~\ref{sec:hyperbolicity}.

\item \emph{Constraint-compatible compressive data.}
Starting from an integral curve of a genuinely nonlinear eigenfield, we add an $O(\varepsilon^2)$ correction that enforces exact neutrality while preserving the leading $O(\varepsilon)$ compression.  Reconstructing $F_0$ from \eqref{intro-constraint} then produces, after the short-scale rescaling in \eqref{notation-scaling}, only a lower-order zero-speed component
\[
w_0(0,\cdot)=O(\varepsilon^2\eta),
\]
whereas the selected compressive mode has size $O(\varepsilon)$.

\item \emph{Stability of the characteristic blow-up mechanism under the constraint.}
We modify the John--H\"ormander--B\"arlin characteristic estimates to accommodate the nonzero Poisson mode.  A crossing estimate for transversal characteristics shows that all nonprincipal modes remain $O(\varepsilon^2)$ on the $O(\varepsilon^{-1})$ scaled time interval.  The principal characteristic derivative therefore satisfies a Riccati inequality with only lower-order errors, which forces finite-time gradient catastrophe.
\end{enumerate}

The model-specific contribution of the gradient-catastrophe argument is to show that the classical genuinely nonlinear steepening mechanism persists after the dynamics are restricted to the Poisson constraint manifold.  The constraint is not discarded or treated as an arbitrary source; it is enforced exactly in the initial data and propagated throughout the evolution, while its dynamical effect appears only through a perturbative zero-speed mode.  This mechanism is qualitatively different from the explicit affine collapse of Theorem~\ref{thm:affine}, where the state itself reaches the boundary $v=0$ of the physical region.

\subsection{Notation and conventions}\label{sec:notation}

We collect here the notation used throughout the paper.  Once the Poisson field has been localized, the five-component state and its equilibrium are written as
\begin{equation}\label{notation-U}
U=(v,u,\theta,q,F)^T,
\qquad
\bar U=(1,0,1,0,0)^T,
\end{equation}
while
\[
W=(v,u,\theta,q)^T,
\qquad
\cC(v,\theta,q)=C_v-\frac{\tau q^2}{\kappa_0\theta^2}
\]
denote, respectively, the fluid--thermal variables and the effective heat-capacity coefficient appearing in the temperature equation.  The physical variables are denoted by $(t,x)$.  In the short-scale blow-up construction we use
\begin{equation}\label{notation-scaling}
x=\ell y,
\qquad
t=\ell s,
\qquad
\ell=\eps\eta,
\end{equation}
where $\eps>0$ measures the amplitude and $\eta>0$ the inverse size of the initial compression; quantities written in $(s,y)$ are understood to be in the scaled variables unless otherwise stated.

For the principal matrix $A(U)$, the characteristic speeds and corresponding right and left eigenvectors are denoted by
\[
\lambda_i(U),\qquad r_i(U),\qquad l_i(U),
\qquad l_i(U)r_j(U)=\delta_{ij},
\]
where $\delta_{ij}$ is the Kronecker symbol.
The indices $i,j,k$ range over the five characteristic families.  The index $0$ is reserved for the zero-speed Poisson family, whereas $p$ denotes a fixed nonzero genuinely nonlinear family selected in the proof of Theorem~\ref{thm:gradient}.  In scaled variables,
\[
L_i:=\partial_s+\lambda_i(U)\partial_y,
\qquad
w_i:=l_i(U)U_y
\]
denote the $i$th characteristic derivative and directional operator.  The symbol $\mathsf S(U)$ is reserved for the symmetrizer in Proposition~\ref{prop:LWP}, while $S(t)$ denotes the characteristic-strip width.  The functionals $J,S,M,V$ introduced in \eqref{Jdef}--\eqref{Vdef} measure, respectively, the $L^1$ strength of the principal wave, the width of its characteristic strip, the state amplitude, and the total transversal-wave contamination.

For the gradient-catastrophe theorem, a classical $H^3$ solution means
\[
U-\bar U\in C([0,T);H^3(\R))\cap C^1([0,T);H^2(\R)),
\]
with values in the physical strictly hyperbolic region under consideration; in one space dimension this regularity gives the pointwise derivatives used in the characteristic argument by Sobolev embedding.  We write $\|\cdot\|_{L^p}=\|\cdot\|_{L^p(\R)}$ and $\|\cdot\|_{H^m}=\|\cdot\|_{H^m(\R)}$, and $|\cdot|$ denotes the Euclidean norm for vectors (and a fixed equivalent norm for matrices).  The symbol $C>0$ denotes a generic constant which may change from line to line; constants carrying subscripts, such as $C_K$, $C_J$, or $c_\lambda$, are fixed once the indicated compact set, profile, or bootstrap parameters have been chosen.  In the small-parameter arguments, $O(\cdot)$ is uniform as $(\eps,\eta)\to(0,0)$ with the physical parameters and the fixed profiles held fixed.  In one-variable asymptotic statements, such as $t\uparrow T_c$, the meaning of $O(\cdot)$ is with all remaining parameters fixed.  The regularity imposed on the equation of state is local in $v$: Theorem~\ref{thm:affine} uses only $a\in C^1((0,1])$, Theorem~\ref{thm:gradient} assumes $a\in C^4$ near $v=1$ (which yields the $C_c^3$ data used in the $H^3$ argument), and Proposition~\ref{prop:LWP} uses $a\in C^{m+1}$ on the compact $v$-interval visited by the local solution.  No regularity away from the relevant state range is needed.  Finally, $\supp f$ denotes the support of $f$; $C_c^k(\R)$ denotes the compactly supported $C^k$ functions, and $C_c^\infty(\R)$ the smooth compactly supported functions on $\R$.

\subsection{Related literature and organization}\label{sec:literature}

The literature most closely related to the present work falls into three groups, and each addresses only part of the mechanism studied here.  First, Cattaneo's law \cite{Cattaneo1948} replaces Fourier's constitutive relation by a relaxation law and is the classical starting point for finite-speed heat-flux models.  General structural results for hyperbolic relaxation systems include entropy-based relaxation theory \cite{ChenLevermoreLiu1994} and critical-threshold results for quasilinear relaxation systems \cite{LiLiu2009}.  For the Euler--Cattaneo--Maxwell system, Kawashima and Ueda \cite{KawashimaUeda2016} established a mathematical entropy, regularity-loss decay, and small-data global stability, while Cao and Tong \cite{CaoTong2026} further studied its decay behavior.  The Cattaneo--Christov literature displays a different structural picture: Angeles \cite{Angeles2022} proved non-hyperbolicity for the multidimensional inviscid coupling, Zhu \cite{Zhu2024} established one-dimensional inviscid well-posedness near equilibrium, and Crin-Barat, Kawashima and Xu \cite{CrinBaratKawashimaXu2025} proved global small-data well-posedness and a relaxation limit for a viscous Cattaneo--Christov approximation.  Hyperbolized or relaxed Navier--Stokes models have also been studied from the viewpoints of local and small-data global existence \cite{RackeSaal2012I,RackeSaal2012II}, Cattaneo heat-conduction limits \cite{HuRacke2016}, and finite-time singularity formation \cite{Schowe2016,HuRackeWang2022,HuRacke2024,Wang2025,Racke2025}.  Recent complementary global theories include the optimal-decay result of Li, Tang and Zhang \cite{LiTangZhang2026} for compressible Navier--Stokes flow with hyperbolic heat conduction.  Most recently, Zhao \cite{Zhao2026HNS} proved gradient catastrophe for a one-dimensional hyperbolic Navier--Stokes system with nonlinear Cattaneo heat conduction and Maxwell-type stress.  These results show that thermal relaxation is compatible with genuinely hyperbolic steepening mechanisms, but they do not contain the Poisson differential constraint that restricts the admissible data in \eqref{CEP}.

Second, finite-time breakdown in Euler--Poisson dynamics has a long history.  The critical-threshold viewpoint---global smoothness versus finite-time breakdown according to the initial configuration---was introduced for Euler--Poisson models by Engelberg, Liu and Tadmor \cite{EngelbergLiuTadmor2001} and was developed in pressure, radial-symmetry, and damping settings in \cite{TadmorWei2008,WeiTadmorBae2012,BhatnagarLiu2020}.  Complementary finite-time blow-up criteria for attractive Euler--Poisson systems were obtained in \cite{ChaeTadmor2008,ChengTadmor2009}.  Of particular relevance, Wang and Chen \cite{WangChen1998} proved finite-time singularity formation for compressible Euler--Poisson fluids with heat diffusion and damping relaxation for smooth data with sufficiently large $C^1$ norm, showing that those dissipative mechanisms do not by themselves preclude breakdown.  The thermal mechanism there is, however, diffusive rather than Cattaneo-hyperbolic.  At the opposite end, Wu and Wu \cite{WuWu2026} obtained global well-posedness and optimal time decay for the \emph{viscous} three-dimensional Navier--Stokes--Poisson system with Cattaneo heat conduction.  Thus the available Poisson literature covers both singular and globally regular regimes, but does not resolve the inviscid, constraint-compatible gradient-catastrophe problem considered here.

Third, the steepening mechanism used in the second main theorem belongs to the classical singularity theory for nonlinear hyperbolic equations.  In one space dimension, John \cite{John1974} and Liu \cite{Liu1979} developed derivative blow-up mechanisms for strictly hyperbolic and genuinely nonlinear wave propagation, building on earlier singularity estimates such as Lax \cite{Lax1964}; H\"ormander \cite{Hormander1987} developed related lifespan estimates, while Sideris \cite{Sideris1985} established a complementary finite-time singularity mechanism for three-dimensional compressible Euler flow.  See also \cite{Dafermos2016} for general background on hyperbolic conservation laws and balance laws.  B\"arlin \cite{Barlin2023} adapted the John--H\"ormander characteristic method to strictly hyperbolic systems with general source terms and proved finite-time derivative blow-up for suitable smooth data.  His theorem is the natural unconstrained comparison result for our problem.  The point at which the present argument departs from that theory is precisely the Poisson compatibility condition: the electric-field component cannot be chosen independently, and an exact pure simple wave need not lie on the constraint manifold.  The construction and estimates in Section~\ref{sec:gradient-proof} show that this obstruction can nevertheless be absorbed into a lower-order zero-speed mode.  To the best of our knowledge, no prior work establishes gradient catastrophe for the specific inviscid one-dimensional closure \eqref{CEP}--\eqref{kappa} while imposing the Poisson compatibility and neutrality constraints exactly on the compressive initial data.  This statement is intentionally model-specific: closely related viscous Poisson--Cattaneo systems and non-Poisson hyperbolic Cattaneo fluid models have been studied in the works cited above.

The remainder of the paper is organized around the two main theorems.  Section~\ref{sec:density-proof} proves the explicit density-collapse result.  Section~\ref{sec:gradient-proof} proves the gradient-catastrophe theorem through localization, compatible data construction, the constrained characteristic bootstrap, and the Riccati argument.  Section~\ref{sec:discussion} concludes with several open problems.

\section{Proof of the density-collapse theorem}\label{sec:density-proof}\label{sec:density-collapse}

We first prove Theorem~\ref{thm:affine}.  The construction is useful not only because it gives an explicit singular solution, but also because it isolates a breakdown mechanism that is different from the genuinely nonlinear steepening studied in Section~\ref{sec:gradient-proof}.

\paragraph{Why an affine ansatz is natural.}
In Lagrangian coordinates, the specific volume is the spatial derivative of the particle map.  A spatially homogeneous specific volume therefore corresponds to an affine deformation of the physical position.  This suggests taking
\[
v(t,x)=b(t),\qquad u(t,x)=b'(t)x.
\]
The Poisson equation then forces the electric field variable $F=\phi_x/v$ to be linear in $x$, while a spatially homogeneous temperature allows the Cattaneo equation to be satisfied with $q=0$.  Thus the full PDE closes on a single scalar ODE for the compression factor $b$.

\begin{proposition}[Affine family]\label{prop:affine-family}
Assume that $a\in C^1$ on an interval containing the range of $b$. Let $b(t)>0$ solve
\begin{equation}\label{b-ode}
b''+b=1.
\end{equation}
Fix $b(0)=b_0>0$ and $\Theta_0>0$, and define
\begin{equation}\label{general-affine}
\begin{aligned}
v(t,x)&=b(t),\\
u(t,x)&=b'(t)x,\\
q(t,x)&=0,\\
\theta(t,x)&=\Theta_0\exp\!\left(\frac{\alpha}{C_v}[b(t)-b_0]\right),\\
F(t,x)&=[1-b(t)]x.
\end{aligned}
\end{equation}
Then \eqref{general-affine}, together with $F=\phi_x/v$, satisfies \eqref{CEP} as long as $b>0$. A corresponding potential is
\begin{equation}\label{general-phi}
\phi(t,x)=\frac12b(t)[1-b(t)]x^2+C(t),
\end{equation}
where $C(t)$ is arbitrary.
\end{proposition}

\begin{proof}
The mass equation gives $b'=u_x$.  Since $v$, $\theta$, and $q$ are independent of $x$, the pressure is a function of $t$ only and hence $p_x=0$.  The momentum equation therefore reduces to
\[
b''x=F=(1-b)x,
\]
which is exactly \eqref{b-ode}.  The Cattaneo equation is automatic because $q=0$ and $\theta_x=0$.  For the energy equation, using $P'(b)=-a(b)$, we obtain
\begin{align*}
e_t+pu_x
&=C_v\theta'+[-a(b)-\alpha]b'
+[a(b)+\alpha(1-\theta)]b'\\
&=C_v\theta'-\alpha\theta b',
\end{align*}
which vanishes by the definition of $\theta$.  Finally,
$F_x=1-b$ and $\phi_x=vF$, so the Poisson equation is satisfied as well.
\end{proof}

The scalar equation \eqref{b-ode} can be solved explicitly:
\begin{equation}\label{bgeneral}
b(t)=1+(b_0-1)\cos t+b_1\sin t,
\qquad b_1:=b'(0).
\end{equation}
Hence an affine solution collapses precisely when the oscillatory part in \eqref{bgeneral} drives $b$ to zero.  The one-parameter family used in Theorem~\ref{thm:affine} is the simplest purely compressive choice.

\begin{proof}[Proof of Theorem \ref{thm:affine}]
Take
\[
b(t)=1-\mu\sin t,
\qquad \mu>1.
\]
Then $b''+b=1$, $b(0)=1$, and the first zero of $b$ is
\[
T_c=\arcsin(\mu^{-1})\in(0,\pi/2).
\]
Proposition~\ref{prop:affine-family} gives the solution stated in Theorem~\ref{thm:affine}.  Moreover,
\[
b'(T_c)=-\mu\cos T_c=-\sqrt{\mu^2-1},
\]
so Taylor expansion at $T_c$ yields
\begin{equation}\label{collapse-rate}
b(t)=\sqrt{\mu^2-1}\,(T_c-t)+O((T_c-t)^2)
\qquad (t\uparrow T_c).
\end{equation}
Consequently the Eulerian density satisfies the first-order asymptotic law
\begin{equation}\label{density-rate}
\rho(t)=\frac1{v(t)}
\sim \frac{1}{\sqrt{\mu^2-1}\,(T_c-t)}.
\end{equation}
This proves the finite-time density collapse.
\end{proof}

\paragraph{Nature of the singularity.}
The collapse in Theorem~\ref{thm:affine} is a loss of the physical state condition $v>0$, not a shock-type gradient catastrophe.  Indeed,
\[
u_x(t,x)=b'(t)=-\mu\cos t
\]
remains finite as $t\uparrow T_c$, while
\[
\theta(t,x)\longrightarrow \exp(-\alpha/C_v)>0,
\qquad q(t,x)\equiv0,
\qquad F(t,x)\longrightarrow x
\]
at every fixed spatial point.  Thus the Jacobian of the Lagrangian deformation vanishes at a finite rate while the compressive velocity gradient stays bounded.  This should be contrasted with Theorem~\ref{thm:gradient}, where $v$ remains uniformly separated from zero and the breakdown is caused by the divergence of first derivatives inside the physical state space.

The construction also shows why the affine solution is largely insensitive to the detailed form of $a(v)$.  Along the ansatz, $p_x=0$, and the identity $P'(v)=-a(v)$ cancels the $a(b)$ contribution in the energy equation.  The density collapse is therefore generated by the interaction of homogeneous compression with the Poisson restoring field, whereas the equation of state enters the genuinely nonlinear characteristic mechanism studied later.

\begin{remark}[Conditional compact localization]\label{rem:affine-localize}
The affine solution itself is not a finite-energy perturbation of equilibrium because $u(t,x)$ grows linearly in $x$.  A compactly supported localization is nevertheless available provided the affine trajectory stays strictly hyperbolic for every $0\le t<T_c$ and the corresponding characteristic cone has finite width up to $T_c$.  More precisely, let
\[
\Lambda_{\rm aff}(t)
:=\max_j |\lambda_j(U_{\rm aff}(t))|
\]
denote the maximal characteristic speed along the affine state and assume
\begin{equation}\label{affine-speed-integrability}
\int_0^{T_c}\Lambda_{\rm aff}(t)\,\dd t<\infty.
\end{equation}
Set
\[
R_c:=\int_0^{T_c}\Lambda_{\rm aff}(t)\,\dd t.
\]
Choose $L>R_c$ and prescribe at $t=0$
\[
v_0=\theta_0=1,\qquad q_0=F_0=0
\]
globally, while $u_0(x)=-\mu x$ on $[-L,L]$ and is smoothly cut off to zero outside a slightly larger interval.  Finite propagation and uniqueness then imply that, as long as both solutions remain classical, the compact-data solution agrees with the affine solution in the central domain bounded by the characteristics issuing from $\pm L$.  In particular, the compact-data solution either loses classical regularity before $T_c$, or its central specific volume follows the affine profile and approaches zero at $T_c$.

The assumptions preceding \eqref{affine-speed-integrability} are essential for this localization argument and are not part of Theorem~\ref{thm:affine}.  Without them, the theorem should be understood as an explicit non-finite-energy collapse solution.  Establishing an unconditional compact-data density-collapse theorem is a separate problem, because the characteristic structure may degenerate as $v\downarrow0$.
\end{remark}

\section{Proof of the gradient-catastrophe theorem}\label{sec:gradient-proof}

We prove Theorem~\ref{thm:gradient} in four steps.  We first localize the Poisson coupling and identify the characteristic structure of the resulting balance law.  We then establish the local continuation theory and construct neutral compressive data lying exactly on the Poisson constraint manifold.  The third step is a characteristic bootstrap showing that the zero-speed Poisson mode and all transversal modes remain lower order.  Finally, the dominant genuinely nonlinear mode is shown to satisfy a Riccati inequality, which yields finite-time derivative blow-up.

\subsection{Localized formulation and characteristic structure}\label{sec:localization}\label{sec:hyperbolicity}

The proof of Theorem~\ref{thm:gradient} starts by separating two roles of the Poisson field. Its spatial derivative is constrained by the density, but after a suitable normalization its time evolution is local. Once this field variable is introduced, the remaining thermodynamic equations form a first-order balance law whose principal part contains four fluid--thermal modes and one zero-speed Poisson mode. We carry out these two reductions before turning to the characteristic calculation.

\medskip
\noindent\textbf{Localizing the Poisson field.}
Set
\begin{equation*}
F=\frac{\phi_x}{v}.
\end{equation*}
Then the Poisson equation \eqref{CEP-e} becomes the differential constraint
\begin{equation}\label{constraint}
F_x=1-v.
\end{equation}
For the compact perturbations used in Theorem~\ref{thm:gradient}, we impose the equilibrium normalization
\[
F(t,-\infty)=u(t,-\infty)=0.
\]
Hence \eqref{constraint} gives
\begin{equation}\label{F-integral-representation}
F(t,x)=\int_{-\infty}^x[1-v(t,y)]\,\dd y.
\end{equation}
Using $v_t=u_x$ and differentiating under the integral yields
\begin{equation}\label{Ft}
F_t=-u.
\end{equation}
Thus the Poisson equation no longer has to be solved at each time: it survives as an invariant spatial constraint, while $F$ evolves by the local equation \eqref{Ft}. If also $F(t,+\infty)=0$, then \eqref{constraint} is equivalent to the neutrality condition
\[
\int_{\R}(v(t,x)-1)\,\dd x=0,
\]
which explains the compatibility condition imposed on the compressive data below.

\begin{lemma}[Propagation of the Poisson constraint]\label{lem:constraint}
Suppose $v,u,F$ satisfy
\[
v_t-u_x=0,\qquad F_t=-u.
\]
Then
\[
\partial_t(F_x+v-1)=0.
\]
Consequently, \eqref{constraint} holds throughout the classical lifespan whenever it holds initially.
\end{lemma}

\begin{proof}
Directly,
\[
\partial_t(F_x+v-1)=(F_t)_x+v_t=-u_x+u_x=0.
\]
\end{proof}

\medskip
\noindent\textbf{Reduction to a local balance law.}
The choice $\kappa=\kappa_0v$ is especially convenient here. The Cattaneo law becomes
\begin{equation}\label{q-local}
\tau q_t+q+\kappa_0\theta_x=0,
\end{equation}
and the constitutive relations reduce to
\begin{align*}
p&=a(v)+\alpha(1-\theta)-\frac{\tau q^2}{2\kappa_0v\theta},\\
e&=C_v\theta+P(v)-\alpha(v-1)+\frac{\tau q^2}{\kappa_0\theta}.
\end{align*}
Define
\begin{equation*}
\cC(v,\theta,q)=C_v-\frac{\tau q^2}{\kappa_0\theta^2}.
\end{equation*}
Since
\[
e_v=-a(v)-\alpha,\qquad e_\theta=\cC,
\qquad e_q=\frac{2\tau q}{\kappa_0\theta},
\]
and \eqref{q-local} gives
\[
q_t=-\frac{q}{\tau}-\frac{\kappa_0}{\tau}\theta_x,
\]
the internal-energy equation is equivalent, wherever $\cC>0$, to
\begin{equation*}
\cC\theta_t-
\left(\alpha\theta+\frac{\tau q^2}{2\kappa_0v\theta}\right)u_x
-\frac{2q}{\theta}\theta_x+q_x
=\frac{2q^2}{\kappa_0\theta}.
\end{equation*}
Therefore, with
\[
U=(v,u,\theta,q,F)^T,
\]
the evolution equations take the local balance-law form
\begin{equation}\label{local-system}
U_t+A(U)U_x=G(U),
\end{equation}
where
\begin{equation}\label{A-matrix}
A(U)=
\begin{pmatrix}
0&-1&0&0&0\\
p_v&0&p_\theta&p_q&0\\
0&B&C&D&0\\
0&0&\kappa_0/\tau&0&0\\
0&0&0&0&0
\end{pmatrix},
\end{equation}
with
\begin{align*}
p_v&=a'(v)+\frac{\tau q^2}{2\kappa_0v^2\theta},
& p_\theta&=-\alpha+\frac{\tau q^2}{2\kappa_0v\theta^2},
& p_q&=-\frac{\tau q}{\kappa_0v\theta},\\
B&=-\frac{\alpha\theta+\dfrac{\tau q^2}{2\kappa_0v\theta}}{\cC},
& C&=-\frac{2q}{\theta\cC},
& D&=\frac1{\cC},
\end{align*}
and
\begin{equation*}
G(U)=
\begin{pmatrix}
0\\ F\\ \dfrac{2q^2}{\kappa_0\theta\cC}\\ -q/\tau\\ -u
\end{pmatrix}.
\end{equation*}
The principal matrix has the block structure
\begin{equation}\label{blockA}
A(U)=\begin{pmatrix}A_4(v,\theta,q)&0\\0&0\end{pmatrix}.
\end{equation}
This identity is the structural reason that the Poisson correction is perturbative in the characteristic argument: $F$ has speed zero and does not enter the principal fluid--thermal block, although it feeds back through the lower-order source $G$.

\begin{proposition}[Equivalence with the original CEP system]\label{prop:equiv}
Let $v>0$, $\theta>0$, and $\cC>0$. A classical solution $U=(v,u,\theta,q,F)$ of \eqref{local-system} satisfying \eqref{constraint} determines a solution of \eqref{CEP} by
\begin{equation}\label{phi-recover}
\phi(t,x)=\phi(t,x_0)+\int_{x_0}^x v(t,y)F(t,y)\,\dd y.
\end{equation}
Conversely, let $(v,u,\theta,q,\phi)$ be a classical solution of \eqref{CEP}, put $F=\phi_x/v$, and assume the equilibrium far field
\begin{equation*}
F(t,-\infty)=u(t,-\infty)=0
\end{equation*}
together with enough decay to justify differentiation in \eqref{F-integral-representation}. Then $U=(v,u,\theta,q,F)$ solves \eqref{local-system}. These assumptions hold for the compact perturbations used below.
\end{proposition}

\begin{proof}
If $U$ solves the localized system and \eqref{constraint}, then \eqref{phi-recover} gives $\phi_x=vF$, so the momentum and Poisson equations recover \eqref{CEP-b} and \eqref{CEP-e}; the other equations are exactly the reductions above. Conversely, the Poisson equation and the normalization at $-\infty$ give \eqref{F-integral-representation}. Therefore
\[
F_t(t,x)=-\int_{-\infty}^x v_t(t,y)\,\dd y
=-\int_{-\infty}^x u_x(t,y)\,\dd y=-u(t,x),
\]
and the preceding calculation yields \eqref{local-system}. The value $\phi(t,x_0)$ is only the additive potential gauge and has no effect on $F$.
\end{proof}

\paragraph{Field-energy identity.}
Although it is not needed in the characteristic blow-up argument, the localized formulation also identifies the correct conserved field energy. If $E=e+\frac12u^2$, then
\begin{equation*}
E_t+(pu+q)_x=uF.
\end{equation*}
Using $F_t=-u$ gives
\begin{equation*}
\left(E+\frac12F^2\right)_t+(pu+q)_x=0.
\end{equation*}
Accordingly,
\begin{equation*}
\widetilde\chi(t,x)=\int_{-\infty}^x
\left(E(t,\xi)-\bar E+\frac12F(t,\xi)^2\right)\,\dd\xi
\end{equation*}
satisfies $\widetilde\chi_t+pu+q=0$ whenever the far-field flux vanishes.

\medskip
\noindent\textbf{Characteristic structure near equilibrium.}
The preceding reduction separates the two effects of the electric field: the invariant relation \eqref{constraint} restricts admissible data, whereas the principal matrix \eqref{blockA} leaves the propagating characteristic structure entirely in the four fluid--thermal variables. Freezing the principal part at equilibrium yields the strict-hyperbolicity and genuine-nonlinearity conditions used below.

Let
\[
\bar U=(1,0,1,0,0)^T.
\]
Set
\begin{equation*}
s:=-a'(1),\qquad d:=\frac{\kappa_0}{\tau C_v},\qquad h:=\frac{\alpha^2}{C_v}.
\end{equation*}
When $s>0$, define
\begin{equation}\label{xpm}
x_\pm=\frac{s+d+h\pm\sqrt{(s+d+h)^2-4sd}}{2}.
\end{equation}
At equilibrium, the nonzero part of the principal matrix is
\begin{equation*}
A_4(\bar U)=
\begin{pmatrix}
0&-1&0&0\\
-s&0&-\alpha&0\\
0&-\alpha/C_v&0&1/C_v\\
0&0&\kappa_0/\tau&0
\end{pmatrix},
\end{equation*}
where $s=-a'(1)$.

\begin{proposition}[Characteristic speeds]\label{prop:hyper}
Assume $a'(1)<0$. Then \eqref{local-system} is strictly hyperbolic in a neighborhood of $\bar U$. Its equilibrium characteristic speeds are
\begin{equation*}
-c_+,-c_-,0,c_-,c_+,
\qquad c_\pm=\sqrt{x_\pm},
\end{equation*}
where $x_\pm$ are given by \eqref{xpm}. Moreover,
\begin{equation}\label{root-location}
0<x_-<\min\{s,d\}\le\max\{s,d\}<x_+.
\end{equation}
\end{proposition}

\begin{proof}
The characteristic polynomial of $A_4(\bar U)$ is
\[
\lambda^4-(s+d+h)\lambda^2+sd.
\]
Thus $x=\lambda^2$ solves
\[
f(x)=x^2-(s+d+h)x+sd=0.
\]
Its discriminant is
\[
(s+d+h)^2-4sd=(s-d)^2+2h(s+d)+h^2>0.
\]
Both roots are positive since their sum and product are positive. Furthermore,
\[
f(0)=sd>0,\quad f(s)=-hs<0,\quad f(d)=-hd<0,
\]
which gives \eqref{root-location}. The fifth eigenvalue is identically zero by \eqref{blockA}. Smooth separation persists in a neighborhood of equilibrium.
\end{proof}

For a nonzero equilibrium eigenvalue $\lambda=\pm\sqrt{x}$, $x\in\{x_-,x_+\}$, normalize the right eigenvector by $r_v=1$. One finds
\begin{equation}\label{right-eig}
r_\lambda=
\left(
1,-\lambda,\frac{x-s}{\alpha},
\frac{dC_v(x-s)}{\alpha\lambda},0
\right)^T.
\end{equation}
The last component is zero: propagating fluid--thermal waves carry no principal $F$ component.

Define
\begin{equation}\label{Ndef}
\cN(x)=a''(1)(x-d)-d(x-s)+\frac{3d-x}{\alpha}(x-s)^2
\end{equation}
and
\begin{equation*}
\DGN:=\cN(x_-)^2+\cN(x_+)^2.
\end{equation*}

\begin{proposition}[Explicit genuine-nonlinearity coefficient]\label{prop:GN}
Let $\lambda=\pm\sqrt{x}$ with $x=x_\pm$. Then
\begin{equation}\label{GNformula}
\nabla\lambda(\bar U)\cdot r_\lambda(\bar U)
=-\frac{\cN(x)}{2\lambda[2x-(s+d+h)]},
\end{equation}
where $\cN$ is defined in \eqref{Ndef}. Consequently, the pair of families with speed $\pm\sqrt{x}$ is genuinely nonlinear at equilibrium if and only if $\cN(x)\ne0$.
\end{proposition}

\begin{proof}
At $q=0$, the characteristic polynomial of the $4\times4$ block is
\begin{equation*}
\mathscr P(\lambda;v,\theta,0)
=\lambda^4+\left(a'(v)-d-\frac{\alpha^2}{C_v}\theta\right)\lambda^2-a'(v)d.
\end{equation*}
Although the eigenvalue is evaluated at $q=0$, the $q$-derivative cannot be omitted because the right eigenvector has a nonzero $q$ component. A direct differentiation of the full determinant gives, at equilibrium,
\begin{align*}
\mathscr P_\lambda&=2\lambda[2x-(s+d+h)],\\
\mathscr P_v&=a''(1)(x-d),\\
\mathscr P_\theta&=-hx,\\
\mathscr P_q&=\frac{\lambda}{C_v}[2(x-s)-\alpha].
\end{align*}
By implicit differentiation,
\[
\lambda_z=-\frac{\mathscr P_z}{\mathscr P_\lambda}.
\]
Using \eqref{right-eig} and
\begin{equation*}
(x-s)(x-d)=hx,
\end{equation*}
which follows from $f(x)=0$, one obtains
\begin{align*}
\mathscr P_\lambda(\nabla\lambda\cdot r_\lambda)
={}&-a''(1)(x-d)
+\frac{hx(x-s)}{\alpha}
-\frac{d(x-s)}{\alpha}[2(x-s)-\alpha]\\
={}&-\left[a''(1)(x-d)-d(x-s)
+\frac{3d-x}{\alpha}(x-s)^2\right].
\end{align*}
This is \eqref{GNformula}.
\end{proof}

Consequently, condition \eqref{GN-assump} is equivalent to
\begin{equation*}
\DGN>0.
\end{equation*}
The next statement identifies explicitly the exceptional parameter set on which this nondegeneracy fails.  For this purpose set
\begin{equation}\label{Qdef}
\cQ=\alpha(s-d)+2(s-d)^2+h(s+d-h)
\end{equation}
and
\begin{equation}\label{a2star}
a_{2,*}=s+\frac{s\,[2(s-d)+h]}{\alpha}.
\end{equation}

\begin{proposition}[Generic nondegeneracy]\label{prop:generic}
Both propagating pairs fail to be genuinely nonlinear at equilibrium if and only if
\begin{equation*}
\cQ=0,
\qquad
a''(1)=a_{2,*},
\end{equation*}
with \eqref{Qdef}--\eqref{a2star}. In particular, if $\cQ\ne0$, at least one propagating pair is genuinely nonlinear for every value of $a''(1)$. Moreover, if
\[
s\ge d,\qquad h<s+d,
\]
then $\cQ>0$, so at least one propagating pair is genuinely nonlinear independently of $a''(1)$.
\end{proposition}

\begin{proof}
Reduce $\alpha\cN(x)$ modulo the quadratic relation $f(x)=0$. The remainder is affine:
\begin{equation*}
\alpha\cN(x)=A_1x+A_0\qquad\text{for }f(x)=0,
\end{equation*}
where
\begin{align*}
A_1&=\alpha a''(1)-\alpha d+2d^2+dh-2ds-h^2,\\
A_0&=d[-\alpha a''(1)+\alpha s-2ds+hs+2s^2].
\end{align*}
Because $x_-\ne x_+$, the two equations $\cN(x_-)=\cN(x_+)=0$ hold if and only if $A_1=A_0=0$. The equation $A_0=0$ gives
\[
a''(1)=s+\frac{s[2(s-d)+h]}{\alpha}=a_{2,*}.
\]
Substitution into $A_1=0$ yields exactly
\[
\alpha(s-d)+2(s-d)^2+h(s+d-h)=0.
\]
Finally, if $s\ge d$ and $h<s+d$, then each term in
\[
\cQ=\alpha(s-d)+2(s-d)^2+h(s+d-h)
\]
is nonnegative and the last term is strictly positive.
\end{proof}

\medskip
\noindent\textbf{A useful algebraic example.}\label{sec:algebraic-check}

For reference, the equilibrium characteristic roots satisfy
\[
x_-x_+=sd,
\qquad x_-+x_+=s+d+h,
\qquad (x-s)(x-d)=hx.
\]
The denominator in \eqref{GNformula} is never zero because
\[
2x_\pm-(s+d+h)=\pm\sqrt{(s+d+h)^2-4sd}.
\]
Thus all loss of genuine nonlinearity is encoded by the numerator $\cN(x_\pm)$.

As a simple example, take
\[
a(v)=v^{-1}-1,
\qquad C_v=\alpha=\kappa_0=\tau=1.
\]
Then
\[
s=d=h=1,
\qquad a''(1)=2,
\qquad x_\pm=\frac{3\pm\sqrt5}{2},
\]
and
\[
\cQ=1>0.
\]
A direct substitution gives
\[
\cN(x_-)=\frac{3-\sqrt5}{2}>0,
\qquad
\cN(x_+)=\frac{3+\sqrt5}{2}>0,
\]
so all four nonzero characteristic families are genuinely nonlinear at equilibrium.

The structural analysis needed for the blow-up construction is now complete. Under $a'(1)<0$ and assumption~\eqref{GN-assump}, we may fix one genuinely nonlinear family $p$ and a sufficiently small neighborhood of $\bar U$ on which the characteristic speeds remain separated and the sign of $\nabla\lambda_p\cdot r_p$ is unchanged. The remaining task is therefore dynamical: construct small physical data on the Poisson constraint manifold that are compressive in the $p$-family, and show that the constraint-generated modes remain lower order until the principal derivative blows up.

\subsection{Local theory and Poisson-compatible compressive data}\label{sec:lwp}\label{sec:compression}

The local theory is used here only as a framework for the singularity argument.  It has two roles: it supplies a classical solution on which the characteristic decomposition is legitimate, and it identifies the only possible breakdown mechanism once the state is known to remain in a compact strictly hyperbolic region.  The actual singularity formation will come from the compatible compressive data constructed below, not from the Sobolev estimate itself.

Let $\mathscr D$ denote the physical region in which
\begin{equation*}
v>0,\qquad \theta>0,\qquad \cC>0,
\end{equation*}
and the five eigenvalues of $A(U)$ are real and simple.  Fix a small neighborhood $\OO\Subset\mathscr D$ of $\bar U$.

\begin{proposition}[Local theory and continuation]\label{prop:LWP}
Assume $a'(1)<0$, let $m\ge3$, and let $a\in C^{m+1}$ on the relevant $v$-interval.  After shrinking $\OO$ if necessary, the balance law \eqref{local-system} is symmetrizable hyperbolic on $\OO$.  More precisely, there is a $C^m$ symmetric positive definite matrix $\mathsf S(U)$ such that $\mathsf S(U)A(U)$ is symmetric.  If
\[
U_0-\bar U\in H^m(\R),\qquad U_0(x)\in K_0\Subset\OO,
\]
then \eqref{local-system} has a unique solution
\begin{equation}\label{LWP-space}
U-\bar U\in C([0,T];H^m)\cap C^1([0,T];H^{m-1})
\end{equation}
for some $T>0$.  If $F_{0x}=1-v_0$, the Poisson constraint persists, and compact perturbations have finite propagation speed.

If the range of a smooth solution is contained in a compact set $K\Subset\OO$, then its symmetrized $H^m$ energy satisfies
\begin{equation}\label{Hm-energy}
\frac{\dd}{\dd t}\mathcal E_m(t)
\le C_K\bigl(1+\|U_x(t)\|_{L^\infty}\bigr)\mathcal E_m(t),
\end{equation}
where
\begin{equation*}
\mathcal E_m(t)
=\sum_{j=0}^m\int_\R
(\partial_x^j(U-\bar U))^T \mathsf S(U)\partial_x^j(U-\bar U)\,\dd x.
\end{equation*}
Consequently, if $T_{\max}<\infty$ while the state remains in some compact set $K\Subset\OO$, then
\begin{equation}\label{continuation-blow}
\int_0^{T_{\max}}\|U_x(t)\|_{L^\infty}\,\dd t=+\infty.
\end{equation}
\end{proposition}

\begin{proof}
Strict hyperbolicity on a sufficiently small contractible neighborhood of $\bar U$ yields a $C^m$ eigenvector matrix $\mathsf R(U)$ satisfying
\[
A(U)\mathsf R(U)=\mathsf R(U)\Lambda(U),
\]
with $\Lambda$ real diagonal.  Put $\mathsf L=\mathsf R^{-1}$ and
\[
\mathsf S(U)=\mathsf L(U)^T\mathsf L(U).
\]
Then $\mathsf S$ is symmetric positive definite and $\mathsf S A=\mathsf L^T\Lambda\mathsf L$ is symmetric.  On every compact $K\Subset\OO$,
\begin{equation}\label{symm-equivalence}
c_K|Z|^2\le Z^T\mathsf S(U)Z\le C_K|Z|^2.
\end{equation}
The quasilinear symmetric-hyperbolic theory therefore yields \eqref{LWP-space}; see \cite{Kato1975}.  Lemma~\ref{lem:constraint} propagates the Poisson constraint.  Since $G(\bar U)=0$, compact perturbations have finite propagation relative to the equilibrium state.

For completeness, write $Z=U-\bar U$ and $Z_j=\partial_x^jZ$.  Differentiating
\[
Z_t+A(U)Z_x=G(U)
\]
gives
\[
\partial_tZ_j+A(U)\partial_xZ_j
=\partial_x^jG(U)-[\partial_x^j,A(U)]Z_x.
\]
Multiplication by $Z_j^T\mathsf S(U)$ and integration use the symmetry of $\mathsf S A$ to give
\[
-\int_\R Z_j^T\mathsf S A\,\partial_xZ_j\,\dd x
=\frac12\int_\R Z_j^T\partial_x(\mathsf S A)Z_j\,\dd x.
\]
Because $\mathsf S=\mathsf S(U)$ and $U_t=-A(U)U_x+G(U)$,
\[
\|\mathsf S_t\|_{L^\infty}\le C_K(1+\|U_x\|_{L^\infty}).
\]
The one-dimensional Moser and commutator estimates give
\[
\|G(U)\|_{H^m}\le C_K\|Z\|_{H^m},\qquad
\|[\partial_x^j,A(U)]Z_x\|_{L^2}
\le C_K\|U_x\|_{L^\infty}\|Z\|_{H^m}.
\]
Summing over $0\le j\le m$ and using \eqref{symm-equivalence} proves \eqref{Hm-energy}.  Hence
\begin{equation*}
\|U(t)-\bar U\|_{H^m}^2
\le C_K\|U_0-\bar U\|_{H^m}^2
\exp\!\left(C_Kt+C_K\int_0^t\|U_x(s)\|_{L^\infty}\,\dd s\right).
\end{equation*}
If the integral of $\|U_x\|_{L^\infty}$ stays finite and the state remains in $K$, the $H^m$ norm stays uniformly bounded.  Restarting the local theory at times $t_n\uparrow T_{\max}$ then gives a uniform positive extension time, contradicting maximality.  This proves \eqref{continuation-blow}.
\end{proof}

\begin{remark}[Breakdown alternatives]\label{rem:breakdown-alternatives}
The continuation statement above is deliberately local: it is used only while the solution stays in the fixed strictly hyperbolic neighborhood $\OO$.  The solution constructed below remains in an $O(\eps)$ neighborhood of $\bar U$ up to its classical lifespan, hence in a compact set $K\Subset\OO$ once $\eps$ is sufficiently small.  In particular, $v$ and $\theta$ stay positive, $\cC$ stays bounded away from zero, and characteristic speeds do not collide.  Proposition~\ref{prop:LWP} therefore reduces finite-time breakdown for the constructed solution to growth of first derivatives; no global continuation criterion on all of $\mathscr D$ is needed.
\end{remark}

\medskip
\noindent\textbf{Why an unconstrained simple wave is not admissible.}
For an unconstrained strictly hyperbolic balance law, a natural starting point is a pure $p$-wave parametrized by an integral curve of $r_p$, with all transversal characteristic modes initially zero.  Here that construction is not physical.  Indeed, if the fifth component were kept identically zero, then $F_{0x}=0$, whereas the Poisson constraint requires
\[
F_{0x}=1-v_0.
\]
Except for the trivial state $v_0\equiv1$, a pure four-dimensional simple wave therefore lies outside the Poisson constraint manifold.  Moreover, if one reconstructs $F_0$ from the constraint, compact support at both spatial infinities requires the neutrality condition
\begin{equation*}
\int_\R(v_0-1)\,\dd x=0.
\end{equation*}
The point of the next construction is to impose both conditions exactly while perturbing the principal $p$-wave only at quadratic order in its amplitude.  This is the model-specific step that allows the unconstrained characteristic blow-up mechanism to survive the Poisson coupling.

\medskip
\noindent\textbf{Poisson-compatible compression on the constraint manifold.}

Fix a genuinely nonlinear nonzero family $p$.  Starting from the normalization used in Proposition~\ref{prop:GN}, flip the common sign of $r_p,l_p$ if necessary so that
\begin{equation}\label{orientation}
\gamma_*:=-\nabla\lambda_p(\bar U)\cdot r_p(\bar U)>0.
\end{equation}
After this sign choice set
\begin{equation*}
\sigma_p:=(r_p)_v(\bar U)\in\{-1,1\}.
\end{equation*}
Thus we do \emph{not} impose the incompatible additional normalization $(r_p)_v(\bar U)=1$ after orienting the field.  Let $W=(v,u,\theta,q)$ denote the four fluid--thermal variables. Because $A$ is independent of $F$ in its principal block, let $\Psi_p(\sigma)$ be the integral curve of the four-dimensional right eigenvector:
\begin{equation*}
\Psi_p'(\sigma)=\widehat r_p(\Psi_p(\sigma)),
\qquad \Psi_p(0)=(1,0,1,0).
\end{equation*}
Then the $v$ component $V_p$ satisfies
\begin{equation}\label{Vp-expand}
V_p(\sigma)=1+\sigma_p\sigma+O(\sigma^2).
\end{equation}

Choose $\alpha_0,\psi\in C_c^\infty((-1/2,1/2))$ such that
\begin{equation}\label{profiles}
\int_\R\alpha_0(y)\,\dd y=0,
\qquad
\int_\R\psi(y)\,\dd y\ne0,
\qquad
A_*:=\max_y\alpha_0'(y)>0.
\end{equation}

\begin{lemma}[Neutral simple-wave correction]\label{lem:neutral-correction}
For all sufficiently small $\eps>0$, there exists $c_\eps=O(\eps^2)$ such that
\begin{equation*}
h_\eps(y)=\eps\alpha_0(y)+c_\eps\psi(y)
\end{equation*}
satisfies
\begin{equation}\label{neutral-y}
\int_\R[V_p(h_\eps(y))-1]~\dd y=0.
\end{equation}
\end{lemma}

\begin{proof}
Define
\[
\mathcal M_\eps(c)=\int_\R[V_p(\eps\alpha_0+c\psi)-1]\,\dd y.
\]
By \eqref{Vp-expand} and $\int\alpha_0=0$, $\mathcal M_\eps(0)=O(\eps^2)$, whereas
\[
\partial_c\mathcal M_\eps(0)=\int_\R V_p'(\eps\alpha_0(y))\psi(y)\,\dd y
=\sigma_p\int_\R\psi(y)\,\dd y+O(\eps),
\]
which is nonzero for small $\eps$. The implicit function theorem gives $c_\eps=O(\eps^2)$.
\end{proof}

Introduce a second small parameter $0<\eta\ll1$ and the narrow length scale
\begin{equation*}
\ell=\eps\eta.
\end{equation*}
In scaled variables
\[
y=x/\ell,\qquad s=t/\ell,
\]
the balance law becomes
\begin{equation}\label{scaled-system}
U_s+A(U)U_y=\ell G(U)=\eps\eta G(U),
\end{equation}
and the Poisson constraint becomes
\begin{equation*}
F_y=\eps\eta(1-v).
\end{equation*}
Define scaled initial data by
\begin{equation}\label{scaled-data}
W_0(y)=\Psi_p(h_\eps(y)),
\qquad
F_0(y)=\eps\eta\int_{-\infty}^y[1-v_0(z)]\,\dd z.
\end{equation}
Because of \eqref{neutral-y}, $F_0$ is compactly supported whenever $v_0-1$ is.

Let $l_i(U),r_i(U)$ be $C^3$ dual left and right eigenvectors on the chosen neighborhood and define characteristic derivative amplitudes
\begin{equation*}
w_i=l_i(U)U_y.
\end{equation*}
The zero-speed mode can be normalized by
\[
r_0=e_5,\qquad l_0=e_5^T.
\]

\begin{lemma}[Size of the initial characteristic modes]\label{lem:initial-modes}
The data \eqref{scaled-data} satisfy
\begin{align}
w_p(0,y)&=h_\eps'(y)=\eps\alpha_0'(y)+O(\eps^2),\label{wp0}\\
w_j(0,y)&=0,\qquad j\ne p,0,\label{wj0}\\
w_0(0,y)&=F_{0y}=\eps\eta(1-v_0)=O(\eps^2\eta).\label{w00}
\end{align}
Moreover,
\begin{equation*}
\|U_0-\bar U\|_{L^\infty}=O(\eps),
\qquad
\|F_0\|_{L^\infty}=O(\eps^2\eta).
\end{equation*}
In physical variables, the amplitude is $O(\eps)$, the support width is $O(\eps\eta)$, and the dominant initial gradient is $O(\eta^{-1})$.
\end{lemma}

\begin{proof}
Differentiating the integral curve gives
\[
(W_0)_y=h_\eps'\widehat r_p(W_0),
\]
which proves \eqref{wp0}--\eqref{wj0}. Equation \eqref{w00} follows from the constraint and $v_0-1=O(\eps)$. Rescaling $\partial_x=\ell^{-1}\partial_y$ gives the last statement.  Explicitly, the physical initial data are
\[
U_0^{\rm ph}(x)=U_0^{\rm sc}(x/\ell),
\]
so $F_{0x}^{\rm ph}=\ell^{-1}F_{0y}^{\rm sc}=1-v_0^{\rm ph}$ and
$\int_\R(v_0^{\rm ph}-1)\,\dd x=\ell\int_\R(v_0^{\rm sc}-1)\,\dd y=0$.
\end{proof}

The separation of scales is now explicit.  In the scaled variables the principal compressive mode has size $O(\eps)$, the constraint-generated zero mode has size only $O(\eps^2\eta)$, all other transversal modes vanish initially, and the source itself carries the small prefactor $\eps\eta$.  The next step is to show that this hierarchy persists on the natural steepening time scale $s=O(\eps^{-1})$: the non-$p$ modes must remain $O(\eps^2)$ while the $p$-mode is allowed to grow.  Once this is established, genuine nonlinearity reduces the final step to a scalar Riccati comparison.

\subsection{Characteristic bootstrap for the constrained wave}\label{sec:barlin}

We next propagate the scale separation obtained at the end of Subsection~\ref{sec:compression}.  If the fifth component were absent, the initial profile would be an exact $p$-simple wave and the John--H\"ormander--B\"arlin estimates would apply in their standard form.  The Poisson constraint changes only one feature of that picture: it inserts a zero-speed mode of size $O(\eps^2\eta)$.  The purpose of this subsection is to show that this extra mode, together with all other transversal waves generated by the source and by nonlinear interactions, stays at the $O(\eps^2)$ level on the steepening time scale $s=O(\eps^{-1})$.

There are four quantities to control.  The functional $J$ measures the $L^1$ strength of the principal $p$-wave inside its characteristic strip; $S$ measures the width of that strip; $M$ measures the total amplitude of the state; and $V$ measures every component that can contaminate the principal wave, namely the non-$p$ characteristic modes together with the part of the $p$-wave that has escaped the original strip.  The bootstrap will preserve
\begin{equation*}
J=O(\eps),\qquad S=O(1),\qquad M=O(\eps),\qquad V=O(\eps^2).
\end{equation*}
The last estimate is the decisive one: it keeps all transversal interactions one order below the compressive mode and is precisely where the Poisson compatibility correction has to be checked.

Let $\Omega$ be a sufficiently small compact neighborhood of $\bar U$, contained in the interior of $\OO$, on which the eigenvalues and dual eigenvectors are $C^3$ and the eigenvalues remain simple.  Fix the eigenvector normalizations $l_i r_j=\delta_{ij}$.  We record the characteristic identities explicitly, since the cancellation in the two-form equation is used repeatedly below.

Write $D A(U)[\zeta]$ for the directional derivative of $A$ in the direction $\zeta$.  Differentiating \eqref{scaled-system}, using $U_y=\sum_k w_k r_k(U)$ and $w_i=l_i(U)U_y$, gives
\begin{equation}\label{characteristic-source-exact}
L_iw_i=\sum_{j,k}\gamma_{ijk}(U)w_jw_k
+\eps\eta\sum_kH_{ik}(U)w_k,
\end{equation}
where
\begin{align*}
\gamma_{ijk}(U)
&=(\lambda_i-\lambda_j)\,D l_i(U)[r_j(U)]\,r_k(U)
-l_i(U)\bigl(D A(U)[r_j(U)]\,r_k(U)\bigr),\\
H_{ik}(U)
&=D l_i(U)[G(U)]\,r_k(U)+l_i(U)D G(U)r_k(U).
\end{align*}
A direct exterior differentiation yields
\begin{equation}\label{two-form-exact}
d\!\left(w_i(\dd y-\lambda_i(U)\dd s)\right)
=\left(\sum_{j,k}\Gamma_{ijk}(U)w_jw_k
+\eps\eta\sum_kH_{ik}(U)w_k\right)\dd s\wedge\dd y,
\end{equation}
with
\begin{equation}\label{Gamma-explicit}
\Gamma_{ijk}(U)
=\gamma_{ijk}(U)+\delta_{ij}\,D\lambda_i(U)[r_k(U)].
\end{equation}
In particular, the source coefficient is the same in \eqref{characteristic-source-exact} and \eqref{two-form-exact}.  The John--H\"ormander cancellation is now immediate:
\begin{equation}\label{Gamma-cancel}
\Gamma_{ijj}(U)=0
\qquad\text{for every }i,j.
\end{equation}
Indeed, if $i=j$, then
\[
\gamma_{iii}=-l_i(D A[r_i]r_i)=-D\lambda_i[r_i],
\]
so \eqref{Gamma-explicit} gives $\Gamma_{iii}=0$.  If $i\ne j$, differentiate $A r_j=\lambda_j r_j$ in the direction $r_j$ and use $l_i r_j=0$ to obtain
\[
l_i(D A[r_j]r_j)=(\lambda_j-\lambda_i)l_i(D r_j[r_j]).
\]
Differentiating $l_i r_j=0$ in the same direction gives
$D l_i[r_j]r_j=-l_i(D r_j[r_j])$, and hence $\gamma_{ijj}=0$.  This proves \eqref{Gamma-cancel}; see also \cite{John1974,Hormander1987,Barlin2023}.

Set
\begin{align}
 c_\lambda&:=\min_{i\ne p}\inf_{U\in\Omega}|\lambda_i(U)-\lambda_p(U)|>0,\notag\\
 \gamma&:=\max_i\sup_{U\in\Omega}\sum_{j,k}|\gamma_{ijk}(U)|,\qquad
 \Gamma:=\max_i\sup_{U\in\Omega}\sum_{j,k}|\Gamma_{ijk}(U)|,\notag\\
 G_*&:=\max_i\sup_{U\in\Omega}\sum_k|H_{ik}(U)|,\qquad
 r_*:=\sup_{U\in\Omega}\sum_k|r_k(U)|,\notag\\
 \Lambda_*&:=\sup_{U\in\Omega}\lambda_{\max}(U)-\inf_{U\in\Omega}\lambda_{\min}(U).\label{Lambda-star}
\end{align}
The scaled source in \eqref{scaled-system} is $\eps\eta G(U)$, hence every source contribution below carries the factor $\eps\eta$.

For later use, let
\[
A_1:=\int_\R|\alpha_0'(y)|\,\dd y.
\]
Since $c_\eps=O(\eps^2)$, after decreasing $\eps$ we have
\begin{equation}\label{J0-bound}
J(0)=\int_{-1/2}^{1/2}|h_\eps'(y)|\,\dd y
\le 2A_1\eps.
\end{equation}
Likewise Lemma \ref{lem:initial-modes} gives a constant $C_P>0$, independent of $\eps,\eta$, such that
\begin{equation}\label{zero-initial-precise}
\max_{i\ne p}\|w_i(0)\|_{L^\infty}
=\|w_0(0)\|_{L^\infty}
\le C_P\eps^2\eta.
\end{equation}

For $a_p(s)=X_p(s;-1/2)$ and $b_p(s)=X_p(s;1/2)$ define the $p$-strip
\[
\mathcal R_p(t)=\{(s,y):0\le s\le t,\ a_p(s)\le y\le b_p(s)\}.
\]
The four bootstrap quantities announced above are defined by
\begin{align}
J(t)&:=\sup_{0\le s\le t}\int_{a_p(s)}^{b_p(s)}|w_p(s,y)|\,\dd y,\label{Jdef}\\
S(t)&:=\sup_{0\le s\le t}[b_p(s)-a_p(s)],\notag\\
M(t)&:=\sup_{0\le s\le t,\ y\in\R}|U(s,y)-\bar U|,\notag\end{align}
and
\begin{align}
\widetilde V(t)&:=\max_{i\ne p}\sup_{0\le s\le t,y\in\R}|w_i(s,y)|,\notag\\
W_p^{\rm out}(t)&:=\sup_{\substack{0\le s\le t,\ y\in\R\\(s,y)\notin\mathcal R_p(t)}}|w_p(s,y)|,\notag\\
V(t)&:=\widetilde V(t)+W_p^{\rm out}(t).\label{Vdef}
\end{align}
Thus $J$ and $S$ describe the principal strip, $M$ keeps the coefficients in a fixed strictly hyperbolic neighborhood, and $V$ collects all waves that are lower order in the desired hierarchy.

\begin{lemma}[Absolute two-form estimate]\label{lem:absolute-twoform}
Fix a characteristic family $i$.  Let $\tau$ be a $C^1$ arc transverse to the $i$-characteristics such that the $i$-characteristic label is monotone along $\tau$ (equivalently, each $i$-characteristic meets $\tau$ at most once).  Let $\mathcal A_i(\tau)$ be the region bounded by $\tau$, the two $i$-characteristics through its endpoints, and the corresponding interval $I_0$ on $\{s=0\}$.  As long as the solution takes values in $\Omega$,
\begin{align}
\int_{\tau}\bigl|w_i(\dd y-\lambda_i\dd s)\bigr|
&\le \int_{I_0}|w_i(0,y)|\,\dd y \nonumber\\
&\quad+\iint_{\mathcal A_i(\tau)}\left|\sum_{j,k}\Gamma_{ijk}w_jw_k+\eps\eta\sum_kH_{ik}w_k\right|\,\dd s\,\dd y.
\label{absolute-twoform}
\end{align}
\end{lemma}

\begin{proof}
Set $\omega_i=w_i(\dd y-\lambda_i\dd s)$.  Equation~\eqref{two-form-exact} gives $d\omega_i=R_i\,\dd s\wedge\dd y$, where $R_i$ is the integrand on the right-hand side of \eqref{absolute-twoform}.  If $w_i$ has one sign on $\tau$, multiply $\omega_i$ by that sign and apply Stokes' theorem.  The integrals over the two lateral $i$-characteristics vanish because $\dd y-\lambda_i\dd s=0$ there.  Taking absolute values therefore gives \eqref{absolute-twoform}.

If $w_i$ changes sign on $\tau$, decompose the relatively open set $\{w_i\neq0\}\cap\tau$ into its connected components.  Apply the preceding constant-sign argument first to finite unions of compactly contained components and then pass to an exhaustion.  The monotonicity of the characteristic label along $\tau$, together with order preservation of the characteristic flow, implies that disjoint components determine disjoint intervals on $\{s=0\}$ and disjoint characteristic regions.  Their bottom and area contributions consequently add without multiplicity.  Monotone convergence yields \eqref{absolute-twoform}.  This is the absolute-value form of the H\"ormander transversal-arc estimate; compare \cite[Lemma~2.1]{Barlin2023}.
\end{proof}

\begin{lemma}[Transversal crossing estimate]\label{lem:crossing}
Assume the solution takes values in $\Omega$, and let $i\ne p$.  Let $X_i(\cdot;z)$ be an $i$-characteristic and let $\omega_i\subset[0,t]$ be the (possibly empty) set of times for which $(s,X_i(s;z))$ belongs to the $p$-strip.  Then $\omega_i$ is an interval, the corresponding arc is transversal to the $p$-characteristics, and
\begin{align}
\int_0^t|w_p(s,X_i(s;z))|\,\dd s
&\le W_p^{\rm out}(t)t+c_\lambda^{-1}\Bigl[J(0)\nonumber\\
&\quad+t\{\Gamma V(t)+\eps\eta G_*\}
\{V(t)S(t)+J(t)\}\Bigr].
\label{crossing-lemma-est}
\end{align}
\end{lemma}

\begin{proof}
As long as the solution is classical, the $p$-characteristic flow $\xi\mapsto X_p(s;\xi)$ is an increasing diffeomorphism onto its image.  Let $\Xi_p(s,y)$ denote its inverse label, so that $X_p(s;\Xi_p(s,y))=y$.  Differentiating this identity shows that along an $i$-characteristic,
\[
\frac{\dd}{\dd s}\Xi_p(s,X_i(s;z))
=(\partial_y\Xi_p)(s,X_i(s;z))\,[\lambda_i(U)-\lambda_p(U)].
\]
Here $\partial_y\Xi_p>0$, and the bracket has a fixed sign and absolute value at least $c_\lambda$ because the characteristic families remain ordered on $\Omega$.  Thus the $p$-label is strictly monotone along the $i$-characteristic.  The times for which that label belongs to $[-1/2,1/2]$ therefore form a single interval (possibly empty), and the corresponding arc meets the $p$-characteristic foliation transversally.  Outside that interval the integrand is bounded by $W_p^{\rm out}(t)$.

On the interior arc $\tau_i$, one has $\dd y=\lambda_i\dd s$, and therefore
\[
|w_p|\,\dd s
\le c_\lambda^{-1}|w_p(\dd y-\lambda_p\dd s)|.
\]
Apply Lemma~\ref{lem:absolute-twoform} with characteristic family $p$ and transversal arc $\tau_i$.  Because the $p$-label is strictly monotone along $\tau_i$, the characteristic region generated by every sign component lies inside the $p$-strip, and distinct components correspond to disjoint label intervals.  Hence the bottom boundary contributes at most $J(0)$.  By $\Gamma_{pjj}=0$, every nonzero quadratic term in the area integrand either contains one factor $w_p$ and one factor bounded by $V(t)$, or contains two factors bounded by $V(t)$.  The source term is bounded by $\eps\eta G_*(|w_p|+V(t))$.  The $p$-strip portion of each time slice has length at most $S(t)$, while the $L^1$ strength of the $p$-wave is bounded by $J(t)$.  Thus the area integral is bounded by
\[
t\{\Gamma V(t)+\eps\eta G_*\}\{V(t)S(t)+J(t)\}.
\]
Adding the outside contribution proves \eqref{crossing-lemma-est}.  This is the geometric estimate used in B\"arlin's proof of Lemma~3.2, with the scaled source $\eps\eta G$ in place of $\eps\kappa g$.
\end{proof}

\begin{lemma}[Constraint-compatible characteristic bootstrap]\label{lem:modified-Barlin}
Fix $T>0$. There exist constants $C_J,C_M,C_S,C_V>0$ and $\nu>0$, depending only on $T$, the system in $\Omega$, and the fixed profiles $\alpha_0,\psi$, such that the following holds.  If $0<\eps,\eta\le\nu$ and a classical solution of \eqref{scaled-system} with initial data \eqref{scaled-data} exists on
\[
0\le s\le T_1\le T\eps^{-1},
\]
then for every $0\le t\le T_1$,
\begin{equation}\label{bootstrap-result}
J(t)<C_J\eps,\qquad
M(t)<C_M\eps,\qquad
S(t)<C_S,\qquad
V(t)<C_V\eps^2.
\end{equation}
In particular, for $\nu$ sufficiently small the solution remains in $\Omega$ throughout this time interval.
\end{lemma}

\begin{proof}
The proof is an open--closed argument.  It has four coupled estimates, arranged in the order suggested by the geometry: first the strip width $S$, then the $L^1$ strength $J$ of the principal wave, then the amplitude $M$, and finally the transversal quantity $V$.  The first three estimates are the familiar simple-wave bounds.  The fourth is the only step modified by Poisson compatibility, through the initial contribution $w_0(0)=O(\eps^2\eta)$.

We follow the quantitative scheme of Lemma~3.2 in \cite{Barlin2023}, recording that additional term explicitly.  Because the coefficient bounds above are taken on $\Omega$, first introduce the exit time
\[
\tau_{\Omega}:=\sup\{\,s\in[0,T_1]: U(\sigma,y)\in\Omega
\text{ for every }0\le\sigma\le s,\ y\in\R\,\}.
\]
The initial data lie in the interior of $\Omega$ for sufficiently small $\eps$.  All estimates below are first carried out on $[0,\tau_{\Omega})$.  We choose the bootstrap amplitude constant and then $\nu$ so that the improved bound $M<C_M\eps$ keeps the range a fixed positive distance from $\partial\Omega$; the open--closed argument will then imply $\tau_{\Omega}=T_1$.

Let $c_*$ be a fixed constant dominating $1$ and the coefficients appearing in the derivative of the $p$-strip width, and choose the bootstrap constants in the order
\begin{align}
C_J&:=4A_1,\notag\\
C_S&:=2c_*(1+C_JT),\label{CS-choice}\\
C_V&:=4\left[1+C_P+G_*c_\lambda^{-1}C_J(1+G_*T)\right],\label{CV-choice}\\
C_M&:=2r_*\left[C_J+C_V\{1+T\Lambda_*\}\right].\label{CM-choice}
\end{align}
All these constants are independent of $\eps,\eta$.  Shrinking $\Omega$ only at the start of the argument, fix $d_{\Omega}>0$ so that the $d_{\Omega}$-neighborhood of $\bar U$ is contained in the interior of $\Omega$, and reduce $\nu$ so that $C_M\nu<d_{\Omega}/2$.  By \eqref{J0-bound},
\[
J(0)\le \frac12 C_J\eps.
\]
Moreover $S(0)=1$, while \eqref{zero-initial-precise} gives
\begin{equation*}
V(0)\le C_P\eps^2\eta.
\end{equation*}
For sufficiently small $\nu$ these inequalities are strict relative to the bounds in \eqref{bootstrap-result}.

Assume that the four bounds \eqref{bootstrap-result} hold on $[0,t)$ for some $t<\tau_{\Omega}$.  We show that all four improve at time $t$.

\smallskip
\noindent\emph{Step 1: control of the strip width.}
The first estimate prevents the two boundary $p$-characteristics from separating on the $s=O(\eps^{-1})$ time scale.  Since $a_p$ and $b_p$ are $p$-characteristics,
\[
\frac{\dd}{\dd s}(b_p-a_p)
=\lambda_p(U(s,b_p))-\lambda_p(U(s,a_p)).
\]
Using $\partial_y\lambda_p=\sum_k (\nabla\lambda_p\cdot r_k)w_k$ and setting $c_{ppk}=\nabla\lambda_p\cdot r_k$ gives
\[
\frac{\dd}{\dd s}(b_p-a_p)
=\int_{a_p(s)}^{b_p(s)}\sum_k c_{ppk}(U)w_k\,\dd y,
\]
whence
\begin{equation*}
S(t)\le 1+c_*\{\widetilde V(t)S(t)t+J(t)t\}.
\end{equation*}
Using $t\le T\eps^{-1}$ and the bootstrap assumptions,
\begin{equation*}
S(t)\le 1+c_*C_JT+c_*TC_V\eps S(t).
\end{equation*}
After decreasing $\nu$ so that $c_*TC_V\nu\le1/3$, one obtains
\[
S(t)\le \frac32(1+c_*C_JT)<C_S
\]
by the choice \eqref{CS-choice} (and, if necessary, replacing $c_*$ once at the start by a larger fixed coefficient bound).

\smallskip
\noindent\emph{Step 2: control of the principal $L^1$ strength.}
The quantity $J$ is chosen because the self-interaction of the $p$-wave cancels from the H\"ormander two-form identity.  Consequently its growth is driven only by transversal waves and the scaled source.  The differentiated system has the characteristic form \eqref{characteristic-source-exact}; for later reference we set
\begin{equation*}
L_i:=\partial_s+\lambda_i(U)\partial_y.
\end{equation*}
Apply Lemma~\ref{lem:absolute-twoform} with $i=p$ to the horizontal top segment $\tau_s=\{s\}\times[a_p(s),b_p(s)]$.  Its lateral boundaries are $p$-characteristics, so their one-form integrals vanish.  By \eqref{Gamma-cancel}, the quadratic area integrand contains no square $w_j^2$.  Inside the strip, every surviving term is therefore bounded either by $\Gamma\widetilde V(t)|w_p|$ or by $\Gamma\widetilde V(t)^2$, while the source term is bounded by $\eps\eta G_*(|w_p|+\widetilde V(t))$.  Integrating over the strip and taking the supremum over $0\le s\le t$ yields
\begin{equation*}
J(t)\le J(0)
+\Gamma\{\widetilde V(t)^2S(t)t+\widetilde V(t)J(t)t\}
+\eps\eta G_*\{\widetilde V(t)S(t)t+J(t)t\}.
\end{equation*}
Substitution of the bootstrap bounds and $t\le T\eps^{-1}$ yields
\begin{align}
J(t)
&\le \frac12C_J\eps
+\Gamma C_V^2C_ST\eps^3
+\eta G_*C_VC_ST\eps^2\nonumber\\
&\quad+\{\Gamma C_VT\eps+\eta G_*T\}J(t).
\label{J-close-exact}
\end{align}
Choose $\nu$ so small that
\begin{equation*}
\Gamma C_VT\nu+G_*T\nu\le\frac16,
\qquad
\Gamma C_V^2C_ST\nu^2+G_*C_VC_ST\nu^2\le\frac18C_J.
\end{equation*}
Then \eqref{J-close-exact} implies
\[
J(t)\le\frac34C_J\eps<C_J\eps.
\]

\smallskip
\noindent\emph{Step 3: control of the state amplitude.}
Once $J=O(\eps)$ and $V=O(\eps^2)$ are available, the amplitude follows from spatial integration of $U_y=\sum_iw_ir_i$.  This step keeps the solution inside the fixed hyperbolic neighborhood $\Omega$.  Because $G(\bar U)=0$, finite propagation for the shifted variable $U-\bar U$ gives a spatial support interval of length at most
\[
1+\Lambda_*t,
\]
with $\Lambda_*$ defined in \eqref{Lambda-star}.
Since $U_y=\sum_iw_ir_i(U)$ and the perturbation vanishes at $y=-\infty$,
\begin{equation*}
M(t)\le r_*\left[J(t)+V(t)\{1+\Lambda_*t\}\right].
\end{equation*}
Hence, using the improved $J$ estimate,
\[
M(t)\le r_*\left[C_J+C_V\eps+C_VT\Lambda_*\right]\eps.
\]
With the choice \eqref{CM-choice}, $M(t)<\frac12C_M\eps$.  Shrinking $\nu$ once more ensures $C_M\nu$ is smaller than the radius of $\Omega$.

\smallskip
\noindent\emph{Step 4: closure of the transversal hierarchy.}
This is the decisive estimate.  We must show that the non-$p$ modes and the escaped part of the $p$-wave stay at order $O(\eps^2)$, despite the nonzero zero-speed Poisson component present initially.  First consider a $p$-characteristic that stays outside $\mathcal R_p(t)$.  Since $w_p(0)=0$ outside $(-1/2,1/2)$, integration of \eqref{characteristic-source-exact} gives
\begin{equation}\label{Wout-rigorous}
W_p^{\rm out}(t)
\le t\{\gamma V(t)+\eps\eta G_*\}V(t)
\le T\{\gamma C_V\eps+G_*\eta\}V(t).
\end{equation}
Indeed, along such a characteristic all characteristic derivatives are bounded by $V(t)$, while $t\le T\eps^{-1}$ and the bootstrap gives $V(t)\le C_V\eps^2$.

Now let $i\ne p$ and integrate \eqref{characteristic-source-exact} along the $i$-characteristic through an arbitrary point.  The relation $\gamma_{ijj}=-\delta_{ij}\langle D\lambda_i,r_i\rangle$ implies that every quadratic term contains at least one factor bounded by $V(t)$ once $i\ne p$.  The only nonzero initial datum among the $i\ne p$ modes is the Poisson zero mode, bounded by \eqref{zero-initial-precise}. Consequently,
\begin{equation}\label{wi-pre-cross}
|w_i(s,y)|
\le C_P\eps^2\eta
+\{\gamma V(t)+\eps\eta G_*\}
\left[tV(t)+\int_0^t|w_p(\sigma,X_i(\sigma;z_i))|\,\dd\sigma\right].
\end{equation}
Lemma~\ref{lem:crossing} gives
\begin{align}
\int_0^t|w_p(\sigma,X_i(\sigma;z_i))|\,\dd\sigma
&\le V(t)t+c_\lambda^{-1}\Bigl[J(0)\nonumber\\
&\quad+t\{\Gamma V(t)+\eps\eta G_*\}
\{V(t)S(t)+J(t)\}\Bigr].
\label{crossing-rigorous}
\end{align}
Before inserting the bootstrap sizes, it is useful to keep the complete algebraic structure visible.  Adding \eqref{Wout-rigorous} and \eqref{wi-pre-cross}, then using \eqref{crossing-rigorous}, gives the master inequality
\begin{align}
V(t)\le{}&C_P\eps^2\eta
+3t\{\gamma V(t)+\eps\eta G_*\}V(t)
+c_\lambda^{-1}\{\gamma V(t)+\eps\eta G_*\}J(0)\nonumber\\
&+c_\lambda^{-1}t\{\gamma V(t)+\eps\eta G_*\}
\{\Gamma V(t)+\eps\eta G_*\}
\{V(t)S(t)+J(t)\}.
\label{V-master}
\end{align}
Every contribution to the $V$-estimate appears in \eqref{V-master}: the first term is the initial Poisson mode, the factor $3t$ consists of one escaped-$p$ contribution and the two $tV$ contributions in \eqref{wi-pre-cross}, the $J(0)$ term is the bottom boundary of the crossing region, and the last line is its two-form area contribution.

Using the bootstrap bounds in \eqref{V-master} gives
\begin{equation}\label{V-final-exact}
V(t)\le C_P\eps^2\eta+Q(\eps,\eta)V(t)
+\eta G_*c_\lambda^{-1}C_J(1+\eta G_*T)\eps^2,
\end{equation}
where one may take
\begin{align}
Q(\eps,\eta)
={}&3T(\gamma C_V\eps+G_*\eta)\nonumber\\
&+\gamma\eps c_\lambda^{-1}
\left[C_J+T(\Gamma C_V\eps+G_*\eta)(C_VC_S\eps+C_J)\right]\nonumber\\
&+G_*\eta c_\lambda^{-1}T\eps
\left[\Gamma(C_VC_S\eps+C_J)+G_*C_S\eta\right].
\label{Q-explicit}
\end{align}
The small-parameter bookkeeping in \eqref{Q-explicit} is worth making explicit.  The first line comes from the escaped $p$-wave and the two copies of the $tV$ term in \eqref{wi-pre-cross}; the second line contains the contribution of the initial principal strength $J(0)$ and the part of the crossing-area term carrying a transversal factor $V$; the third line comes from the source-driven part of that area term.  Thus every coefficient multiplying $V(t)$ contains at least one factor $\eps$ or $\eta$.  More quantitatively, for $0<\eps,\eta\le1$ one may take
\begin{align*}
C_Q:={}&3T(\gamma C_V+G_*)
+\gamma c_\lambda^{-1}\Bigl[C_J
+T(\Gamma C_V+G_*)(C_VC_S+C_J)\Bigr]\nonumber\\
&+G_*c_\lambda^{-1}T
\Bigl[\Gamma(C_VC_S+C_J)+G_*C_S\Bigr],
\end{align*}
so that
\begin{equation*}
Q(\eps,\eta)\le C_Q(\eps+\eta).
\end{equation*}
The terms in \eqref{V-final-exact} which do not multiply $V$ satisfy, again for $0<\eta\le1$,
\begin{equation}\label{V-additive-small}
C_P\eps^2\eta
+\eta G_*c_\lambda^{-1}C_J(1+\eta G_*T)\eps^2
\le A_0\eps^2,
\end{equation}
where
\[
A_0:=C_P+G_*c_\lambda^{-1}C_J(1+G_*T)
<\frac14C_V
\]
by \eqref{CV-choice}.  This displays the two relevant scales directly: the feedback coefficient is $O(\eps+\eta)$, whereas all forcing terms independent of $V$ are $O(\eps^2\eta)$ and hence lie below the $O(\eps^2)$ bootstrap scale.

Reduce $\nu$ so that $C_Q(\eps+\eta)\le1/4$.  Equations \eqref{V-final-exact} and \eqref{V-additive-small} then give
\[
V(t)\le\frac14V(t)+\frac14C_V\eps^2,
\]
therefore
\begin{equation*}
V(t)\le\frac13C_V\eps^2<C_V\eps^2.
\end{equation*}
This is the only place where the proof differs from the exact-simple-wave argument: the additional term $C_P\eps^2\eta$ in \eqref{V-final-exact} comes from the Poisson compatibility condition.  It lies strictly below the $O(\eps^2)$ bootstrap scale.

All four estimates are strict improvements.  The quantities $J,M,S,V$ are continuous and nondecreasing in the terminal time, so the usual open--closed argument gives the bounds up to $\tau_{\Omega}$.  The estimate $M<C_M\eps<d_{\Omega}/2$ prevents the range from reaching $\partial\Omega$; hence $\tau_{\Omega}=T_1$.  This proves \eqref{bootstrap-result} on the full interval.  In particular, the Poisson contribution is lower order in the precise sense that
\[
w_0(0)=O(\eps^2\eta),\qquad |\lambda_p-\lambda_0|\ge c_\lambda,\qquad V(t)=O(\eps^2).
\]
Thus the constraint-generated mode is both one order smaller than the principal compression and uniformly transversal to it.
\end{proof}

The bootstrap has now established the only hierarchy needed for singularity formation: the solution amplitude remains $O(\eps)$ and every mode except the principal one remains $O(\eps^2)$.  Hence the quadratic self-interaction of the genuinely nonlinear $p$-family dominates all coupling and source terms along a suitable $p$-characteristic.  We turn to this final Riccati step.

\subsection{Riccati growth and completion of the proof}\label{sec:riccati}

The previous subsection yields the decisive scale separation
\[
M=O(\eps),\qquad V=O(\eps^2),
\]
throughout a scaled time interval of length $O(\eps^{-1})$.  Hence the selected genuinely nonlinear family sees, to leading order, only its quadratic self-interaction.  The rest of the argument has four steps: persistence of the initial compression, a Riccati lower bound, comparison with an explicit scalar equation, and finally the continuation argument in physical variables.

Before starting the comparison, we fix the time window once and for all.  The family $p$ and the profile $\alpha_0$ have already been fixed, hence so have $\gamma_*>0$ and $A_*>0$.  Set
\begin{equation*}
K_0:=\frac{16}{3\gamma_*A_*},
\end{equation*}
and choose a constant
\begin{equation}\label{T-choice-riccati}
T>K_0.
\end{equation}
From this point on, Lemma~\ref{lem:modified-Barlin} is invoked with this fixed $T$; consequently its constants $C_J,C_M,C_S,C_V$ and its smallness threshold $\nu$ are fixed before $\eps$ and $\eta$ are chosen.

Let $y_*$ satisfy
\[
\alpha_0'(y_*)=A_*>0
\]
and let $X_p(s;y_*)$ be the $p$-characteristic issued from $y_*$.  Set
\begin{equation*}
\mathcal W(s):=w_p(s,X_p(s;y_*)).
\end{equation*}
By \eqref{wp0} and $c_\eps=O(\eps^2)$,
\begin{equation}\label{W0}
\mathcal W(0)=A_*\eps+O(\eps^2),
\qquad
\mathcal W(0)\ge \frac12A_*\eps
\end{equation}
for all sufficiently small $\eps$.  Moreover, the orientation \eqref{orientation} and the estimate $M=O(\eps)$ from Lemma~\ref{lem:modified-Barlin} imply
\begin{equation*}
\gamma_{ppp}(U)\ge\frac12\gamma_*
\end{equation*}
throughout the bootstrap interval.  Therefore \eqref{characteristic-source-exact} gives
\begin{equation}\label{W-pre-Riccati}
\mathcal W'
\ge\frac12\gamma_*\mathcal W^2
-\gamma\{V^2+V|\mathcal W|\}
-\eps\eta G_*\{V+|\mathcal W|\}.
\end{equation}

\begin{lemma}[Persistence of compression and Riccati lower bound]\label{lem:compression-persistence}
There exist $\eps_0,\eta_0>0$ such that, for $0<\eps\le\eps_0$ and $0<\eta\le\eta_0$, every classical solution covered by Lemma~\ref{lem:modified-Barlin} satisfies
\begin{equation}\label{W-persistent}
\mathcal W(s)\ge \mathcal W(0)>0
\end{equation}
and
\begin{equation}\label{Riccati-rigorous}
\mathcal W'
\ge a\mathcal W^2-b\eps\eta\mathcal W,
\qquad
a:=\frac38\gamma_*,\quad b:=2G_*.
\end{equation}
\end{lemma}

\begin{proof}
Assume first that $\mathcal W\ge\mathcal W(0)/2$ on an interval beginning at $s=0$.  From \eqref{W0} and Lemma~\ref{lem:modified-Barlin},
\[
\mathcal W\ge \frac14A_*\eps,
\qquad
V\le C_V\eps^2
\le c_W\eps\mathcal W,
\qquad
c_W:=\frac{4C_V}{A_*}.
\]
Substitution into \eqref{W-pre-Riccati} yields
\[
\mathcal W'
\ge
\Bigl[\frac12\gamma_*-\gamma c_W\eps(1+c_W\eps)\Bigr]\mathcal W^2
-\eps\eta G_*(1+c_W\eps)\mathcal W.
\]
Reduce $\eps_0$ so that
\[
\frac12\gamma_*-\gamma c_W\eps_0(1+c_W\eps_0)
\ge \frac38\gamma_*,
\qquad
1+c_W\eps_0\le2.
\]
Then \eqref{Riccati-rigorous} follows.  Since $\mathcal W\ge A_*\eps/4$, the right-hand side of \eqref{Riccati-rigorous} is strictly positive once
\begin{equation}\label{eta-small-rigorous}
0<\eta\le\eta_0<\frac{3\gamma_*A_*}{64G_*},
\end{equation}
with no restriction from \eqref{eta-small-rigorous} when $G_*=0$.  Hence $\mathcal W$ cannot decrease to $\mathcal W(0)/2$.  A first-exit argument proves \eqref{W-persistent}, and the same estimate then holds on the whole classical interval under consideration.
\end{proof}

To obtain finite-time growth, compare with the scalar Riccati equation.  Let $Y$ solve
\begin{equation}\label{comparison-ode}
Y'=aY^2-b\eps\eta Y,
\qquad
Y(0)=\frac12A_*\eps.
\end{equation}
By \eqref{W0}, Lemma~\ref{lem:compression-persistence}, and the scalar comparison principle,
\begin{equation}\label{comparison-WY}
\mathcal W(s)\ge Y(s)
\end{equation}
as long as the classical solution exists.  If $G_*>0$, the smallness condition imposed in Lemma~\ref{lem:compression-persistence} is already stronger than
\begin{equation*}
\eta<\frac{aA_*}{2b}=\frac{3\gamma_*A_*}{32G_*},
\end{equation*}
so the comparison equation is supercritical from the initial time.  Its blow-up time is
\begin{equation*}
s_b=K(\eta)\eps^{-1},\qquad
K(\eta):=
\frac{1}{b\eta}
\log\!\left(\frac{aA_*/2}{aA_*/2-b\eta}\right).
\end{equation*}
Moreover,
\begin{equation}\label{Keta-limit}
K(\eta)=K_0+O(\eta),\qquad K(\eta)\longrightarrow K_0
\quad(\eta\downarrow0).
\end{equation}
When $G_*=0$, equation \eqref{comparison-ode} reduces to $Y'=aY^2$, and we set $K(\eta)\equiv K_0$, so again $s_b=K_0\eps^{-1}$.

Because $T>K_0$ by \eqref{T-choice-riccati}, continuity in \eqref{Keta-limit} gives a number $\eta_T>0$ such that
\begin{equation}\label{etaT-choice}
0<\eta\le\eta_T
\quad\Longrightarrow\quad
K(\eta)<T.
\end{equation}
Hence, after imposing this final smallness condition on $\eta$, one has
\begin{equation*}
s_b<T\eps^{-1},
\end{equation*}
so the whole Riccati comparison lies inside the bootstrap window of Lemma~\ref{lem:modified-Barlin}.

Returning to physical variables, $t=\ell s$ with $\ell=\eps\eta$, the comparison scale $s_b=O(\eps^{-1})$ corresponds to a physical time of order $\eta$.  In particular, the lifespan estimate proved below will have the form
\begin{equation*}
T_{\max}\le C\eta.
\end{equation*}
More precisely, if $G_*>0$,
\begin{equation*}
T_{\max}
\le
\frac1b\log\!\left(\frac{aA_*/2}{aA_*/2-b\eta}\right)
=K_0\eta+O(\eta^2),
\end{equation*}
while for $G_*=0$ one has exactly $T_{\max}\le K_0\eta$ at the comparison scale.  Thus the state amplitude can be made arbitrarily small through $\eps$, whereas the physical initial compression is of size $\eta^{-1}$ and can be made arbitrarily strong by narrowing the pulse.

\begin{proof}[Proof of Theorem \ref{thm:gradient}]
By condition \eqref{GN-assump}, at least one nonzero characteristic family is genuinely nonlinear.  Fix such a family and orient it by \eqref{orientation}.  The profile $\alpha_0$ has already been fixed in \eqref{profiles}, and the constant $T$ has been fixed by \eqref{T-choice-riccati}.  Invoke Lemma~\ref{lem:modified-Barlin} with this $T$, and let $\nu=\nu(T)$ be its smallness threshold.  Next let $\eps_0,\eta_0$ be the constants from Lemma~\ref{lem:compression-persistence}, and let $\eta_T$ be given by \eqref{etaT-choice}.

Now construct the neutral Poisson-compatible profile from Lemmas~\ref{lem:neutral-correction} and \ref{lem:initial-modes}.  Since its state amplitude is $O(\eps)$ uniformly for $0<\eta\le1$, first choose
\[
0<\eps\le\min\{\nu,\eps_0\}
\]
small enough that the physical initial perturbation has $L^\infty$ norm below the prescribed $\delta_0$.  After $\eps$ has been fixed, choose
\begin{equation*}
0<\eta\le\min\{\nu,\eta_0,\eta_T,1\}.
\end{equation*}
Thus every bootstrap, persistence, and comparison smallness condition has been imposed in a noncircular order before the evolution is considered.  Scale the data back to physical variables with $\ell=\eps\eta$.

Apply Proposition~\ref{prop:LWP} with $m=3$. Let $T_{\max}$ be the maximal physical $H^3$ lifespan and set $S_{\max}=T_{\max}/\ell$.  If $S_{\max}>s_b$, then the solution is classical throughout $[0,s_b]$ in scaled variables.  Lemma~\ref{lem:modified-Barlin} keeps the state in an $O(\eps)$ neighborhood of equilibrium there, while \eqref{comparison-WY} forces $w_p$ to diverge by $s=s_b$, a contradiction.  Therefore
\begin{equation*}
T_{\max}\le \ell s_b\le C\eta.
\end{equation*}

It remains to identify the mechanism of breakdown.  For every $t<T_{\max}$, Lemma~\ref{lem:modified-Barlin} applies on the corresponding compact scaled interval, so the state stays in a fixed compact set $K\Subset\OO$ inside the physical strictly hyperbolic region.  In particular, $v$ and $\theta$ stay uniformly positive, $\cC=C_v-\tau q^2/(\kappa_0\theta^2)$ stays away from zero, and strict hyperbolicity persists.  The continuation criterion in Proposition~\ref{prop:LWP} therefore yields
\[
\int_0^{T_{\max}}\|U_x(t)\|_{L^\infty}\,\dd t=+\infty.
\]
On the other hand, the propagated Poisson constraint and the $M=O(\eps)$ estimate give
\[
F_x=1-v=O(\eps)
\]
uniformly on $[0,T_{\max})$.  Hence the loss of regularity must occur in the fluid--thermal derivatives:
\[
\limsup_{t\uparrow T_{\max}}
\bigl(\|v_x(t)\|_\infty+\|u_x(t)\|_\infty+
\|\theta_x(t)\|_\infty+\|q_x(t)\|_\infty\bigr)=+\infty.
\]
This proves \eqref{grad-blow-main} and the asserted bound on $F_x$, whether the maximal solution survives until the comparison time or breaks down earlier.

If the maximal scaled lifespan reaches the comparison threshold, $S_{\max}=s_b$, then the selected characteristic derivative itself diverges as $s\uparrow s_b$.  Because the nonzero left eigenvectors have vanishing $F$ component by the block structure \eqref{A-matrix}, this divergence is purely fluid--thermal.  Moreover,
\begin{equation*}
w_p^{\rm ph}=\ell^{-1}w_p^{\rm sc},
\qquad \ell=\eps\eta,
\end{equation*}
so blow-up of the scaled characteristic derivative is exactly blow-up of the corresponding physical derivative.

Finally, the perturbation has the equilibrium far field.  Proposition~\ref{prop:equiv} therefore reconstructs a potential $\phi$ and identifies the localized solution with a solution of the original CEP system.  The proof is complete.
\end{proof}

\section{Concluding remarks}\label{sec:discussion}

We have identified two distinct finite-time breakdown mechanisms for the one-dimensional Euler--Poisson--Cattaneo system.  The affine family of Theorem~\ref{thm:affine} reaches the boundary of the physical state space through
\[
 v(t)\downarrow0,
 \qquad
 \rho(t)=v(t)^{-1}\uparrow\infty,
\]
while the velocity gradient remains bounded.  By contrast, Theorem~\ref{thm:gradient} produces a small-amplitude solution which stays in a compact strictly hyperbolic neighborhood of equilibrium, with
\[
 v,\theta\ge c_0>0,
 \qquad
 C_v-\frac{\tau q^2}{\kappa_0\theta^2}\ge c_0>0,
\]
but loses $C^1$ regularity in finite time.  Thus density collapse and gradient catastrophe are genuinely different singularity mechanisms in the same model.

The main structural point in the second mechanism is the Poisson constraint.  After introducing
\[
 F=\frac{\phi_x}{v},
\]
the field equation becomes the propagated differential constraint $F_x=1-v$, while, under $a'(1)<0$, the evolution system becomes a local strictly hyperbolic balance law near equilibrium.  The constraint prevents the direct use of an unconstrained pure simple wave.  The construction in Subsection~\ref{sec:compression} restores neutrality and Poisson compatibility exactly, yet introduces only an $O(\eps^2\eta)$ zero-speed mode.  The characteristic estimates in Subsection~\ref{sec:barlin} show that this mode remains below the $O(\eps^2)$ transversal-wave scale, so the genuinely nonlinear Riccati mechanism survives the field coupling.  In this sense, the Poisson interaction changes the admissible data manifold without destroying the short-scale steepening mechanism.

The two results leave several natural questions open.
\begin{enumerate}[label=(\roman*)]
\item \textbf{Critical thresholds and mechanism selection.}  It would be interesting to identify classes of initial data for which the Poisson field separates global smooth evolution from gradient catastrophe, and to understand which data lead instead to the state-space collapse $v\downarrow0$.
\item \textbf{The relaxation limit.}  As $\tau\to0^+$, the Cattaneo law approaches Fourier heat conduction, whereas the characteristic parameter $d=\kappa_0/(\tau C_v)$ becomes singular.  Determining whether the hyperbolic gradient catastrophe persists, disappears, or changes scale in this limit would clarify the transition from hyperbolic to parabolic heat transport.
\item \textbf{Regularization and singular profiles.}  The present argument detects blow-up of a nonzero characteristic derivative and hence of at least one of $v_x,u_x,\theta_x,q_x$.  A sharper description of the asymptotic derivative ratios along the singular characteristic, and of their modification after adding momentum viscosity, would give a more precise picture of the singularity profile.
\end{enumerate}

\section*{Acknowledgments}
Qingsong Zhao was supported by the National Natural Science Foundation of China under Grant Number 12401281.

\end{document}